\documentclass[11pt]{amsart}

\usepackage[colorlinks=true, pdfstartview=FitV, linkcolor=blue, 
      citecolor=blue]{hyperref}
\usepackage{amsfonts, amssymb, amsmath, mdframed, graphicx}

\usepackage{bbm, a4wide}

\DeclareMathAlphabet{\mymathbb}{U}{bbold}{m}{n}

\newcommand{\AAA}{\mathbb{A}}
\newcommand{\RR}{\mathbb{R}}
\newcommand{\ZZ}{\ts\mathbb{Z}}
\newcommand{\NN}{\mathbb{N}}
\newcommand{\QQ}{\mathbb{Q}\ts}
\newcommand{\CC}{\mathbb{C}}
\newcommand{\EE}{\mathbb{E}}
\newcommand{\PP}{\mathbb{P}}
\newcommand{\TT}{\mathbb{T}}
\newcommand{\XX}{\mathbb{X}}
\newcommand{\YY}{\mathbb{Y}}
\newcommand{\SSS}{\mathbb{S}}

\newcommand{\cA}{\mathcal{A}}
\newcommand{\cB}{\mathcal{B}}
\newcommand{\cI}{\mathcal{I}}
\newcommand{\cL}{\mathcal{L}}
\newcommand{\cO}{\mathcal{O}}
\newcommand{\cP}{\mathcal{P}}

\newcommand{\vD}{\varDelta}
\newcommand{\vG}{\varGamma}

\newcommand{\fp}{\mathfrak{p}}
\newcommand{\fq}{\mathfrak{q}}

\newcommand{\fa}{\mathfrak{a}}
\newcommand{\fb}{\mathfrak{b}}

\newcommand{\ts}{\hspace{0.5pt}}
\newcommand{\nts}{\hspace{-0.5pt}}

\newtheorem{theorem}{Theorem}
\newtheorem{prop}[theorem]{Proposition}
\newtheorem{coro}[theorem]{Corollary}
\newtheorem{lemma}[theorem]{Lemma}
\newtheorem{fact}[theorem]{Fact}
\theoremstyle{definition}
\newtheorem{definition}[theorem]{Definition}

\newtheorem{remark}[theorem]{Remark}

\newcommand{\ee}{\ts\mathrm{e}}
\newcommand{\dd}{\,\mathrm{d}\ts}
\newcommand{\card}{\mathrm{card}}
\newcommand{\dens}{\mathrm{dens}}
\newcommand{\den}{\mathrm{den}}
\newcommand{\lcm}{\mathrm{lcm}}
\newcommand{\vol}{\mathrm{vol}}
\newcommand{\ii}{\ts\mathrm{i}\ts}
\newcommand{\bs}{\boldsymbol}
\newcommand{\tr}{\mathrm{tr}}
\newcommand{\trans}{\scriptscriptstyle\mathsf{T}}

\newcommand{\No}{\mathrm{N}}
\newcommand{\pg}{\cP_{\nts\!_{\mathrm{G}}}}
\newcommand{\pe}{\cP_{\!_{\mathrm{E}}}}
\newcommand{\sg}{{\scriptscriptstyle \mathrm{G}}}
\newcommand{\nh}{\nu_{_\mathrm{H}}}

\DeclareMathOperator{\bigtimes}{\raisebox{-0.4ex}{\mbox{\huge $\times$}}}
\DeclareMathOperator{\Bigtimes}{\raisebox{-0.7ex}{\mbox{\Huge $\times$}}}

\newcommand{\exend}{\hfill$\Diamond$}
\newcommand{\defeq}{\mathrel{\mathop:}=}
\newcommand{\eqdef}{=\mathrel{\mathop:}}

\newcommand{\pa}{\phantom{a}}

\newcommand{\ol}[1]{{\ts\overline{\nts #1 \nts}\ts}}

\newcommand{\myfrac}[2]{\frac{\raisebox{-2pt}{$#1$}}
  {\raisebox{0.5pt}{$#2$}}}

\title[Dynamical spectrum of power-free integers]{Dynamical
  spectrum of power-free integers \\[2mm] in quadratic
  number fields and beyond}

\author{Michael Baake}
\address{Fakult\"at f\"ur Mathematik, Universit\"at Bielefeld, \newline
  \indent  Postfach 100131, 33501 Bielefeld, Germany}
\email{$\{$mbaake,dluz$\}$@math.uni-bielefeld.de}

\author{Daniel Luz}

\author{Tanja I. Schindler}
\address{Faculty of Mathematics and Computer Science,
  Jagiellonian University, \newline
\indent Ul.\  {\L}ojasiewicza 6, 30-348 Krakow, Poland, and \newline 
\indent Department of Mathematics and Statistics,
  University of Exeter, \newline
\indent Exeter EX4 4QF, United Kingdom
}
\email{tanja.schindler@uj.edu.pl}
\email{t.schindler@exeter.ac.uk}

\begin{document}

\begin{abstract}
  Power-free integers and related lattice subsets give rise to
  interesting dynamical systems.  They are revisited from a spectral
  perspective, in the setting of the Halmos--von Neumann theorem. With
  respect to the natural patch frequency measure, also known as the
  Mirsky measure, many of these systems have pure-point dynamical
  spectrum, but trivial topological point spectrum. We calculate the
  spectra explicitly, in additive notation, and derive their group
  structure, both for a large class of $\cB$-free lattice systems in
  $\RR^d$ and for power-free integers in quadratic number
  fields. Further, in all cases, the eigenfunctions can be given in
  closed form, via the Fourier--Bohr coefficients of generic elements
  and their translates, which form a subset of full Mirsky measure.
  Based on a simple argument via Kolmogorov's strong law of large
  numbers, we show how the Fourier--Bohr coefficients also provide the
  eigenfunctions for the unique measure of maximal entropy, and that
  we get phase consistency for both measures.
\end{abstract}


\maketitle

\bigskip

\section{Introduction and overview}\label{sec:intro}

The set of square-free integers induces an interesting
\emph{topological dynamical system} (TDS) with pure-point dynamical
spectrum relative to the patch frequency (or Mirsky) measure.
Likewise, the visible (or primitive) points of the integer lattice
$\ZZ^2$ induce a TDS with analogous properties. Both systems have many
generalisations, either in the direction of $\cB$-free systems in one
dimension or towards their analogues among algebraic $\cB$-free
lattice systems in higher dimensions. An interesting class of examples
emerges from considering $\cB$-free systems of integers in algebraic
number fields, as well as certain $\cB$-free lattice systems in the
sense of \cite{BBHLN}.

Some of these systems are re-analysed here from a spectral
perspective, where we employ the cut and project formalism for weak
model sets developed in \cite{Meyer,Schreiber}. To do so, we use the
notion of a \emph{locally compact Abelian group} (LCAG) and refer to
\cite{Reiter} for background on harmonic analysis in this setting.
Our motivation comes from the Halmos--von Neumann theorem for ergodic
dynamical systems with pure-point spectrum. It states that two such
systems are measure-theoretically isomorphic if and only if they have
the same spectrum, thus establishing one of the first cases of a
complete invariant in dynamical systems theory.

The pure-point nature of the systems analysed below can be obtained
as a consequence of identifying them as weak model sets of maximal
density and then using a key result from \cite{BHS}. To actually
profit from spectral information, we need to calculate the spectrum
explicitly, which often is more difficult than merely establishing the
pure-point nature of a system. The point of this paper is that large
and interesting families of number-theoretic dynamical systems exist
where this is actually possible, thus giving an explicit
group-theoretic invariant to tell them apart with respect to
measure-theoretic conjugacy.

The systems we are interested in have many invariant measures.  Two
among them are special, namely the natural patch frequency measure,
which is also known as the Mirsky measure, and the measure of maximal
entropy. With respect to the Mirsky measure, which has zero entropy,
our systems have pure-point dynamical spectrum, but trivial
topological point spectrum. Nevertheless, via a generic element $V$,
we can calculate the eigenfunctions in closed form via the
\emph{Fourier--Bohr} (FB) coefficients of $V$.  The eigenfunctions are
known to be continuous on a set of full measure \cite{BHS,Keller}.

The measure of maximal entropy emerges from a joining (in fact, a
direct product) of the Mirsky measure with the Bernoulli measure of
the fair coin toss, and has mixed spectrum with a pure-point part (and
eigenfunctions that also derive from the FB coefficients of generic
elements) and an absolutely continuous part, with countable Lebesgue
spectrum. We give an easy and constructive proof for the
eigenfunctions with respect to the measure of maximal entropy, and
otherwise concentrate on the spectral theory with respect to the
Mirsky measure.  \smallskip

The paper is organised as follows. First, in Section~\ref{sec:guide},
where we also introduce our notation, we summarise key results of the
visible lattice points and the corresponding dynamical system in
Theorem~\ref{thm:vis-spec} and Corollary~\ref{coro:vis-spec}, which
will serve as our guiding example. We introduce a simplified method for 
the computation of the spectrum in Remark~\ref{rem:dual-alternative},
which will be vital for all later examples. We then construct the
eigenfunctions via the FB coefficients in
Theorem~\ref{thm:vp-spec} and derive the corresponding situation
for the measure of maximal entropy in Corollary~\ref{coro:mixed}, via
an application of Kolmogorov's strong law of large numbers.

An analogous situation is then met for the power-free integers in
Section~\ref{sec:power-free}, as summarised in
Theorem~\ref{thm:k-free-integers}, which suggests that a more general
setting is reasonable and possible. This is derived in
Section~\ref{sec:extend} in the setting of $\cB$-free lattice systems
of Erd\H{o}s type, which can be seen as a general framework for our
approach. Here, the spectrum and its group structure is presented in
Proposition~\ref{prop:gen-latt} and Theorem~\ref{thm:gen-latt}, thus
setting our general scene.

Then, we compute more concrete results for power-free integers in
quadratic number fields in Section~\ref{sec:quadratic}, where we treat
the cases of imaginary and real fields separately. The spectral
results are given in Theorems~\ref{thm:spectrum-imag} and
\ref{thm:spectrum-real}, respectively. This section requires some
ideal-theoretic arguments, though we begin both parts with cases of
class number $1$ that permit working with numbers rather than ideals.

Some underlying material on uniform distribution, which we need for
computing the eigenfunctions, is provided in the Appendix.

\clearpage

\section{A guiding example with many facets}\label{sec:guide}

Let us start with a well-studied example that will be a guide for most
of our generalisations. It emerges from the set of visible points $V$
of the integer lattice $\ZZ^d$, defined by
\[
    V \, \defeq \, \bigl\{ (x^{\pa}_{1}, \ldots , x^{\pa}_{d} ):
    \text{all } x^{\pa}_{i} \in \ZZ \text{ and}\,
    \gcd (x^{\pa}_{1}, \ldots , x^{\pa}_{d}) = 1 \bigr\} \ts ,
\]
where one enforces $d\geqslant 2$ to avoid trivial situations. The set
$V$ has holes of arbitrary size, as follows from an application of the
\emph{Chinese remainder theorem} (CRT). Consequently, though being
uniformly discrete, $V$ is not relatively dense, and thus not a Delone
set; see \cite{TAO} for background material. In fact, since the holes
appear lattice-periodically, no addition of a zero-density set can
turn $V$ into a Delone set.

The translation of $V$ by $t = (t^{\pa}_{1}, \ldots , t^{\pa}_{d})\in\ZZ^d$
is defined via $ t + V = \{ t + x : x \in V \}$, and the translation
orbit closure
\[
    \XX^{\pa}_{V} \, \defeq \, \overline{\ZZ^d + V}
\]
is a compact space, where the closure is taken in the \emph{local
topology}. In this special case of a Fell--Chabauty topology, two point
sets in $\ZZ^d$ are $\varepsilon$-close if they agree in a ball of
radius $1/\varepsilon$ around $0$ (one can equally well use cubes,
which defines the same topology). Then, the translation action of
$\ZZ^d$ on $\XX^{\pa}_{V}$ is continuous, and $(\XX^{\pa}_{V}, \ZZ^d)$ is a
TDS. It is far from being minimal, and $\XX^{\pa}_{V}$ contains the empty
set, which forms the only minimal subsystem. Furthermore, $\XX^{\pa}_{V}$
is subset closed or \emph{hereditary}, which is to say that, for each
$S\in \XX^{\pa}_{V}$, every subset of $S$ also lies in $\XX^{\pa}_{V}$; see
\cite{PH,BBHLN} and references therein for details.

Upon identifying the point set $V$ with its characteristic function,
$\bs{1}^{\pa}_{V}$, one obtains an equivalent representation of
$\XX^{\pa}_{V}$ as a subshift of the full shift, $\{ 0,1\}^{\ZZ^d}$,
where the natural product topology corresponds to the local topology
mentioned above. We shall always identify these two pictures, thus
profiting both from the geometric setting of sets and the symbolic
setting of subshifts. Since
$V = \ZZ^d \setminus \bigcup_{p\in\cP} p \ZZ^d$, we see that $V$
consists of all points from $\ZZ^d$ with the property that the zero
coset modulo $p \ZZ^d$ is missing, for every $p\in\cP$. Consequently,
$\XX^{\pa}_{V}$ must be contained in the subshift $\AAA$ of
\emph{admissible sets}, which are the subsets of $\ZZ^d$ with the
property that, for every $p\in\cP$, at least one coset modulo
$p \ZZ^d$ is missing. In fact, we even have $\XX^{\pa}_{V} = \AAA$,
compare \cite{Abda,BBHLN}, which is a strong property. Such subshifts
are called \emph{admissible}, and all examples discussed below belong
to this class.

The (natural) density of $V$ as a point set in $\RR^d$, denoted as
$\dens (V)$, is $1/\zeta (d)$, where $\zeta$ is Riemann's zeta
function. Here, the term `natural' refers to a sequence of centred,
closed balls $B_r (0)$ with growing radius being used as averaging
sequence, so
\[
    \dens (V) \, = \lim_{r\to\infty}
    \frac{\card \{ x\in V : \| x \| \leqslant r \} }
    {\vol (B_r (0))} \, = \, \myfrac{1}{\zeta (d)}
\]
holds for any $d\geqslant 2$; compare \cite{BMP} and references
therein. One interesting feature is that the set $V$ is an element of
$\XX^{\pa}_{V}$ with maximal density. In fact, $V$ is also maximal in the
sense that no single point can be added to $V$ without kicking it out
of $\XX^{\pa}_{V}$. Note that one can also use centred cubes or other
centred van Hove sequences, see the discussion on tied density in
\cite{BMP}, but the density cannot be uniform, due to the existence of
holes of arbitrary size in $V \!$.

With respect to the balls of growing radius (or to any tied averaging
sequence that is van Hove or F{\o}lner), also the relative patch
frequencies of $V$ exist. Here, a \emph{patch} $P$ is the intersection
of $V$ with a compact set $B$, so $P = V \cap B$, and its (natural)
frequency emerges as the limit of counting the number of occurrences
of $P$ (up to translations) within a large ball $B_r (0)$, divided by
the volume of the ball, in the limit as $r\to\infty$. Using these
frequencies as the measures of the cylinder sets defined by the
patches, they induce the (natural) patch frequency or \emph{Mirsky
  measure}, denoted by $\mu_{_\mathrm{M}}$, which is an invariant
probability measure on $\XX_V$. This measure is ergodic, and $V$ is
generic for it \cite{Keller,BHS}. Then,
$(\XX^{\pa}_{V}, \ZZ^{d}, \mu_{_\mathrm{M}})$ is a
\emph{measure-theoretic dynamical system} (MTDS) with the following
interesting spectral property \cite{BMP,BBHLN}, where
$\TT^d = \RR^d \nts /\ZZ^d = [0,1)^d$ denotes the $d$-torus with
coordinate-wise addition modulo $1$.

\begin{theorem}\label{thm:vis-spec}
  The dynamical system\/ $(\XX^{\pa}_{V},\ZZ^d, \mu_{_\mathrm{M}})$ has
  pure-point spectrum. In additive notation, the spectrum is given by
  the group\/
  $\{ k \in \QQ^d\cap \TT^d : \den (k) \text{ is square-free} \}
  \subset \TT^d$.  The corresponding eigenfunctions are measurable,
  but, except for the trivial eigenfunction, do not have a continuous
  representative. \qed
\end{theorem}

Here, the \emph{denominator} of a vector $k\in\QQ^d$, denoted by
$\den (k)$, is uniquely specified by writing
$k = (k^{\pa}_{1}, \ldots , k^{\pa}_{d})/q$ with all $k_i \in \ZZ$ and
$q\in\NN$ subject to $\gcd(k^{\pa}_{1}, \ldots, k^{\pa}_{d}, q)=1$, which
gives $\den (k) = q$.  In particular, one has $\den (0) = 1$. The
statement on the eigenfunctions means that the system has trivial
topological point spectrum, though it is known that all eigenfunctions
are continuous on a subset of $\XX^{\pa}_{V}$ of full measure
\cite{BHS,Keller}. \smallskip

For some purposes, it is advantageous to also consider $V$ under the
translation action of $\RR^d$. Then, the hull is
$\YY^{\pa}_{\nts V} = \overline{ \{ t + V : t \in\RR^d \} }$, again in
the local topology. It is now modified by saying that two discrete
point sets in $\RR^d$ are $\varepsilon$-close if they agree in
$B_{1/\varepsilon} (0)$, possibly after shifting one of them by a
vector of length at most $\varepsilon$. On the symbolic side, this
corresponds to a standard suspension of $\ZZ^d$ into $\RR^d$; compare
\cite[Ch.~3.4]{VO}. The Mirsky measure extends accordingly, with
cylinder sets being defined by patches up to translations in a small
$\varepsilon$-ball, whose volume then appears as a factor to the
patch frequency.  Now, Theorem~\ref{thm:vis-spec} has the following
obvious counterpart; compare \cite{BMP,TAO,BBHLN,LSS}.

\begin{coro}\label{coro:vis-spec}
  The dynamical system\/ $(\YY^{\pa}_{\nts V},\RR^d, \mu_{_\mathrm{M}})$
  has measurable pure-point spectrum. In additive notation, the
  spectrum is given by the group
\[
  L^{\circledast} \, = \, \{ k \in \QQ^d : \den (k)
  \text{ is square-free} \} \, \subset \, \RR^d,
\]
  while the topological point spectrum is\/ $\ZZ^d$.  \qed
\end{coro}

This extension is also natural when considering $V$ as a weak model
set of maximal density in the sense of \cite{BHS,RS}; see
\cite{Meyer,Schreiber} for origin and background.  This approach is
possible because $x\in \ZZ^d$ is an element of $V$ if and only if, for
every rational prime $p$, the reduction of $x$ modulo $p$ is not the
zero element of $\ZZ^d \nts / p \ZZ^d$. It is now natural to consider
the compact Abelian group
$H = \bigotimes_{p\in\cP} \ZZ^d \nts / p \ZZ^d$, where $\cP$ denotes
the set of rational primes, and define the mapping
$\star : \ZZ^d \longrightarrow H$ by
$x \mapsto x^{\star} = \bigl( x \bmod  p \ZZ^d \bigr)_{p\in\cP}$.
Then, we can characterise $V$ as
\[
  V \,  = \, \{ x \in \ZZ^d : \text{no coordinate of }
  x^{\star} \text{ is } 0 \} 
    \, = \, \{ x \in \ZZ^d : x^{\star} \in W \}
\]
with the coding set or \emph{window}
$W = \bigtimes_{\! p\in\cP} \bigl( (\ZZ^d \nts / p \ZZ^d)\setminus \{
0 \} \bigr)$, which is a compact subset of $H$. This situation is
usually summarised in a \emph{cut and project scheme} (CPS), which we
briefly recall for the example at hand; see \cite{Moody-rev,TAO} for
further background and \cite[Ch.~5a]{Bernd} for an alternative view
via adeles. Here, we have
\begin{equation}\label{eq:vis-cps}
\renewcommand{\arraystretch}{1.2}\begin{array}{r@{}ccccc@{}l}
   & \RR^d & \xleftarrow{\,\;\;\pi\;\;\,} & \nts\RR^d \times H \nts & 
        \xrightarrow{\;\pi^{\pa}_{\mathrm{int}\;}} & H & \\
   & \cup & & \cup & & \cup & \hspace*{-1ex} 
   \raisebox{1pt}{\text{\footnotesize dense}} \\
   & \ZZ^d & \xleftarrow{\; 1-1 \;} & \cL & 
        \xrightarrow{\; \hphantom{1-1} \;} & (\ZZ^d)^{\star} & \\
   & \| & & & & \| & \\
   & L & \multicolumn{3}{c}{\xrightarrow{\qquad\qquad\;\;\star
       \;\;\qquad\qquad}} 
       &  {L_{}}^{\star\nts}  & \\
\end{array}\renewcommand{\arraystretch}{1}
\end{equation}
where $\pi$ and $\pi^{\pa}_{\mathrm{int}}$ are the canonical projections
to \emph{direct space} $(\RR^d)$ and \emph{internal space} ($H$),
respectively. Further, $\cL = \{ (x, x^{\star}) : x \in \ZZ^d \}$ is
the diagonal embedding of $\ZZ^d$ into $\RR^d {\ts\times\ts} H$. Note
that $\cL$ still is a lattice in $\RR^d{\ts\times\ts}H$ (that is, a
co-compact discrete sub\-group).  As $\pi|^{\pa}_{\cL}$ is a bijection
between $\cL$ and its image, the $\star$-map is well defined. From now
on, we denote this kind of CPS by the triple $(\RR^d, H, \cL)$,
variants of which will appear throughout the paper.

The set $W$ is compact, but has empty interior (because $0$ is missing
from all components) and hence consists of boundary only. We are thus
in the realm of \emph{weak model sets}, and need to determine its type
according to \cite{BHS}. If $H$ is equipped with its normalised Haar
measure $\nh$, so $\nh (H) =1$, counting cosets modulo $p$ gives
\[
  \nh (W) \, = \prod_{p\in\cP} \frac{p^d - 1}{p^d} \, =
  \prod_{p\in\cP} \bigl( 1-p^{-d} \bigr) \, = \, \myfrac{1}{\zeta (d)}
  \ts ,
\]
which coincides with $\dens (V)$, where the infinite product converges
absolutely because we have $\sum_{p\in\cP} p^{-d} < \infty$ for
$d\geqslant 2$. Consequently, $V$ is a weak model set of maximal
density, because $\cL$ is a unimodular (co-volume one) lattice in
$\RR^d{\ts\times\ts} H$. Based on this connection, one can prove
Theorem~\ref{thm:vis-spec} and Corollary~\ref{coro:vis-spec} via the
general diffraction formula for weak model sets of maximal density
\cite{BHS,Keller} and the equivalence theorem for dynamical and
diffraction spectra \cite{LMS,BL}. In this setting, the relation
between topological and measure-theoretic aspects is reasonably well
understood; see \cite{Keller} and references therein for details.

There are (at least) two independent ways to compute the dynamical
spectrum, one via the countable supporting set of the pure-point
diffraction measure of $V$ (also known as the Fourier--Bohr spectrum)
and another via the dual CPS \cite{Moody-rev} with the lattice
$\cL^{0}$, which is the annihilator of $\cL$ from the original CPS,
and its projection into $\RR^d$. The first method was used in
\cite{BMP,BH}, and the second in \cite{RS}. Let us summarise the
latter, as we derive a simpler method from it to compute the spectrum
$L^{\circledast}$ and then generalise this approach to other systems.

Recall that the \emph{dual group} of $\RR^d {\ts\times\ts} H$ is
$\widehat{\RR^d{\ts\times\ts} H} \simeq \widehat{\RR^d} {\,\times\,}
\widehat{H}$, where $\widehat{\RR^d} \simeq \RR^d$ is self-dual and
\[
  \widehat{H} \, = \, \bigoplus_{p\in\cP} \ZZ^d/p \ZZ^d
\]
is the direct sum, where only \emph{finitely} many coordinates are
non-zero, by standard results from the harmonic analysis of LCAGs
\cite{Reiter}. The resulting dual CPS is
\begin{equation}\label{eq:vis-cps-dual}
\renewcommand{\arraystretch}{1.2}\begin{array}{r@{}ccccc@{}l}
   & \RR^d & \xleftarrow{\,\;\;\pi\;\;\,}
           & \nts\RR^d \times \widehat{H} \nts & 
        \xrightarrow{\;\pi^{\pa}_{\mathrm{int}\;}} & \widehat{H} & \\
   & \cup & & \cup & & \cup & \hspace*{-1ex} 
   \raisebox{1pt}{\text{\footnotesize dense}} \\
   & \pi (\cL^{0}) & \xleftarrow{\; 1-1 \;} & \cL^{0} & 
             \xrightarrow{\; \hphantom{1-1} \;} &
             \pi^{\pa}_{\mathrm{int}}(\cL^{0}) & \\
   & \| & & & & \| & \\
   & L^{\circledast} &
     \multicolumn{3}{c}{\xrightarrow{\qquad\qquad\;\;\star
       \;\;\qquad\qquad}} 
       &  (L^{\circledast})^{\star\nts}  & \\
\end{array}\renewcommand{\arraystretch}{1}
\end{equation}
which is a CPS of type $(\RR^d, \widehat{H}, \cL^{0})$, with $\cL^{0}$
the annihilator of $\cL$ from the above CPS $(\RR^d, H, \cL)$ in
\eqref{eq:vis-cps}. Usually, the dual group $\widehat{H}$ is defined
as the group of continuous characters on $H$ under pointwise
multiplication.  Here, we use an additive version, where
$k =(k_p)^{\pa}_{p\in\cP}\in\widehat{H}$ stands for the continuous
character $\chi^{\pa}_{k}$ defined by
\[
   y \mapsto \chi^{\pa}_{k} (y) \, \defeq \prod_{p\in\cP} \exp 
   \Bigl( 2 \pi \ii \, \frac{y_p \ts k_p}{p} \Bigr) \qquad
   \text{for } y = (y_p)^{\pa}_{p\in\cP} \in H  .
\]
Note that the character is well defined, since only finitely many
$k_p$ differ from $0$, so the product on the right-hand side is
effectively a finite one.  With this notation, the annihilator
$\cL^{0}$ consists of all points
$(u,k) \in \widehat{\RR^d} \times \widehat{H} \simeq \RR^d \times
\widehat{H}$ such that
\begin{equation}\label{eq:annihilation}
     \exp \bigl( 2 \pi \ii \ts x u \bigr) \prod_{p\in\cP} \exp 
     \Bigl( 2 \pi \ii \frac{x^{\star}_p \ts k^{\pa}_p}{p} \Bigr) \, = \, 1 
\end{equation}
holds for all $x\in\ZZ^d$ or, equivalently, for all
$(x,x^{\star}) \in \cL$.  Here, $xu$ and analogous expressions stand
for the standard inner product of two $d$-dimensional vectors.

By linearity and the rules for calculating modulo $p$, it is clear
that it suffices to satisfy this condition for $x = e_i$ running
through the standard integer basis of $\ZZ^d$. Observing that
$u = (u^{\pa}_{1}, \ldots, u^{\pa}_{d})$ and
$k_p = (k^{\pa}_{1,p} , \ldots , k^{\pa}_{d,p})$,
Eq.~\eqref{eq:annihilation} implies coordinate-wise conditions, namely
\begin{equation}\label{eq:m-diff}
    u^{\pa}_i \, = \, m^{\pa}_{i} - \sum_{p\in\cP} \frac{k^{\pa}_{i,p}}{p}
\end{equation}
with arbitrary $m^{\pa}_{i}\in\ZZ$.  In particular, $(u,k) \in \cL^{0}$
implies $(u+t,k) \in \cL^{0}$ for all $t\in\ZZ^d$, as is clear from
\eqref{eq:annihilation}. Since only finitely many $k_p$ differ from
$0$, one can bring each of these conditions to a form with common
denominator. Putting all conditions together, one can then see that
they are equivalent to
\begin{equation}\label{eq:spec}
  u \in \{ q \in \QQ^d : \den (q) \text{ is square-free} \} \eqdef
  L^{\circledast} ,
\end{equation}
which is the \emph{Fourier--Bohr} (or FB) spectrum of $V \!$. This
also gives the annihilator of $\cL$ as
$\cL^{0} = \{ ( u, u^{\star}) : u \in L^{\circledast} \}$, where the
$\star$-map is that of the dual CPS, as defined by
\[
  L^{\circledast} \nts \ni u \, \longmapsto \, u^{\star} \defeq \bigl(
  - \lcm ( \den (u), p) \ts u \bmod p \bigr)_{p\in\cP} \, ,
\]
where the entry at $p$ is $ - \den (u) \ts u \bmod p$ for all
$p \mid \den(u)$ and $0$ otherwise.  This mapping emerges from taking
all $m^{\pa}_i =0$ in \eqref{eq:m-diff} without loss of generality,
because $(u+t)^{\star} = u^{\star}$ for all $t\in\ZZ^d$.  More
precisely,
$\star \colon L^{\circledast} \xrightarrow{\quad} \widehat{H}$ is a
group homomorphism with kernel $\ZZ^d$.  Strictly speaking, we now
have two $\star$-maps, one for the original CPS and one for its
dual. Nevertheless, we shall use the same symbol for both, as
misunderstandings are unlikely.

\begin{remark}\label{rem:dual-alternative}
  The structure of $L^{\circledast}$ can also be understood and
  calculated as follows. Let $\cP$ be the set of rational primes,
  ordered and numbered increasingly, so
  $\cP = \{ p^{\pa}_{1}, p^{\pa}_{2}, \ldots \}$ with
  $p^{\pa}_{i} < p^{\pa}_{i+1}$ for all $i\in\NN$. We know from the above
  that $V = \ZZ^d \setminus \bigcup_{i=1}^{\infty} p_i \ZZ^d$. If we
  cut out only \emph{finitely} many cosets, for the primes
  $\{ p^{\pa}_{i_1}, \ldots, p^{\pa}_{i_s}\}$ say, we end up with a point
  set that is periodic with $p^{\pa}_{i_1} \! \cdots \ts p^{\pa}_{i_s} \ZZ^d$ 
  as its lattice of periods.

  Relevant for the diffraction of $V$ clearly are all possible
  lattices of the form
\[
     \bigl( p^{\pa}_{i_{1}} \ZZ^d \cap \ldots \cap p^{\pa}_{i_{s}} \ZZ^d 
     \bigr)^{*} \ts = \: \bigl( p^{\pa}_{i_{1}} \ZZ^d \bigr)^{*} + \ldots +
     \bigl( p_{i_{s}} \ZZ^d \bigr)^{*} \ts = \: p^{-1}_{i_{1}} \ZZ^d
     + \ldots + p^{-1}_{i_{s}} \ZZ^d ,
\]
for any choice of the finitely many primes
$p_{i_1}, \ldots , p_{i_s}$, where ${}^*$ denotes the dual of a
lattice. Now, the set of all elements of $\QQ^d$ that lie in some
lattice of this form consists of the rational vectors with square-free
denominators, which are precisely the elements of $L^{\circledast}$.
Conversely, every element $k\in L^{\circledast}$ lies in some lattice
of this form, and we can thus write the set $L^{\circledast}$ as
\[
    L^{\circledast} \ts = \sum_{p\in\cP} p^{-1} \ZZ^d ,
\]
where the elements of $L^{\circledast}$ are
the sums where only \emph{finitely} many terms are non-zero.
Note that $L^{\circledast} $ is a torsion-free Abelian group, here
realised as a subgroup of $\QQ^d$.    \exend
\end{remark}

\begin{remark}\label{rem:dual-discrete}
  Any $y\in L^{\circledast}$ has a unique decomposition $y=k+ t$ with
  $k \in \TT^d = \RR^d / \ZZ^d = [0,1)^d$, which is a fundamental
  domain for $\ZZ^d$, and $t\in\ZZ^d$. Consider now
  $\widetilde{H} = \bigoplus_{p\in\cP} \ts p^{-1} \ZZ^d \nts / \ZZ^d$
  as the direct sum of finite, discrete subgroups of $\TT^d$, where
  $\widetilde{H}=L^{\circledast}\nts \cap \TT^d$ as a set. Under
  addition modulo $1$, it is a subgroup of $\TT^d$, and we have
\[
  L^{\circledast}\! / \ZZ^d \, = \, 
  \bigoplus_{p\in\cP} p^{-1} \ts \ZZ^d \nts / \ZZ^d 
    \, = \, \{ q \in \TT^d : \den (q) \text{ is square-free} \}  \ts .
\]
All summands are finite Abelian groups, with
$p^{-1}\ZZ^d \nts / \ZZ^d\simeq C_{p}^{d}$. The chosen
representation allows to consider them as subgroups of $\TT^d$, which
is best suited for the dynamical interpretation. The elements of
$L^{\circledast}\! / \ZZ^d$ are uniquely represented as the sums where
only \emph{finitely} many (or no) terms are non-zero.  This also
clarifies the group structure of the FB spectrum
$L^{\circledast} \nts$.  Note that $\widetilde{H} \simeq \widehat{H}$,
but we write it in this way to match the structure of $\TT^d$. This is
an important fact because it actually simplifies the computation of
$L^{\circledast}$ considerably. In particular, it allows us to bypass
the somewhat tedious computation using the characters.  \exend
\end{remark}

The set $L^{\circledast}$ is also the \emph{dynamical spectrum} of the
$\RR^d$-flow induced by $V$, compare \cite{BL}, and its calculation
via the dual CPS will be possible more generally. In fact, this is
where Remark~\ref{rem:dual-alternative} will become quite crucial.
Likewise, the group $L^{\circledast}\! / \ZZ^d$ from
Remark~\ref{rem:dual-discrete} is the dynamical spectrum for the
discrete group action of $\ZZ^d$. Another advantage of the connection
with diffraction theory and a CPS is that this setting also allows for
a closed form of the eigenfunctions (relative to the Mirsky measure
$\mu_{_\mathrm{M}}$) as follows. For $y\in \RR^d$, we define the
corresponding \emph{FB coefficient} of $V$ as
\begin{equation}\label{eq:FB-def-1}
    a^{\pa}_{V} (y) \, \defeq \lim_{r\to\infty} \myfrac{1}{\vol (B_r (0))}
    \sum_{ x \in V_r } \ee^{- 2 \pi \ii \ts xy}, \qquad \text{with } 
    V_r = V \cap B_r (0) \ts .
\end{equation}
The limit exists for all $y\in\RR^d$, as can be shown with an explicit
convergence argument. The latter is based on approximating $V$ by a
nested sequence of periodic point sets that are obtained from taking
only the first $N$ primes into account and then letting $N\to\infty$.
This is a natural way to view the limit, as it actually shows how the
coefficients emerge from those of a series of approximating periodic
systems.

The volume-averaged exponential sum in \eqref{eq:FB-def-1} is
nontrivial for any $y\in L^{\circledast}$, which is a countable set,
and vanishes everywhere else \cite{BMP,RS}. Further, one observes the
behaviour under translations as
\[
    a^{\pa}_{t+V} (y) \, = \, \ee^{-2 \pi \ii \ts t y} \ts a^{\pa}_{V} (y) \ts .
\]
So, given $y\in L^{\circledast}$, we have an eigenvector equation
along the translation orbit of $V\!$, with eigenvalue
$\lambda_y = \ee^{-2\pi\ii\ts ty}$ on the unit circle, where $y $
provides the advantageous additive notation mentioned earlier.  The
corresponding definition and eigenvalue equation applies to all other
elements of the hull for which the limit exists, which are most of
them in a measure-theoretic sense, not only for $\mu_{_\mathrm{M}}$,
in ways we will explain later; see \cite{LS,B+} for the required
notions of genericity and their relation to Besicovitch almost
periodicity.

When we consider the $\ZZ^d$-action, it follows from the results of
\cite{BHS,Keller} that, except for the trivial eigenfunction (which
corresponds to $y=0$), no eigenfunction can possess a continuous
representative, but that they are continuous on a subset of
$\XX^{\pa}_{V}$ of full measure. This is related to $\XX^{\pa}_{V}$
being a limit of periodic systems (with pure-point spectrum and
continuous eigenfunctions) in such a way that the spectrum of
$(\XX^{\pa}_{V},\ZZ^d,\mu_{_\mathrm{M}})$ is still obtained as a
limit, but continuity of the eigenfunctions is lost.  Under the
suspension of $\XX^{\pa}_{V}$ to $\YY^{\pa}_{V}$ for the translation
action of $\RR^d$, precisely the eigenfunctions for $y\in \ZZ^d$ are
continuous, but none of the others. Still, it is sufficient to know
them on the translation orbit of sufficiently many elements of the
hull, for instance for the generic elements for $\mu_{_\mathrm{M}}$,
which is a set of full measure. Here, we derive the FB coefficients
for $V$, where we have the following result.

\begin{theorem}\label{thm:vp-spec}
  The FB coefficients of the visible lattice points\/ $V$ are given by
\[
    a^{\pa}_{V} (y) \, = \, \begin{cases} 
       \frac{1}{\zeta(d)} \prod_{p \mid \den (y)}
          \frac{1}{1 - p^d}, & \text{if }\ts
          y \in L^{\circledast} , \\
       0, & \text{otherwise}, \end{cases}
\]  
where\/
$L^{\circledast} = {\sum}_{p\in\cP} \, p^{-1} \ZZ^d$ from 
Remark~$\ref{rem:dual-alternative}$ is the
pure-point dynamical spectrum of the MTDS\/
$(\YY^{\pa}_{\nts V}, \RR^d, \mu_{_\mathrm{M}})$ from
Corollary~$\ref{coro:vis-spec}$. Further, $a^{\pa}_{V} (y)$ represents
the eigenfunction corresponding to\/ $y\in L^{\circledast}$, with
consistent relative phases in the sense of diffraction and the
normalisation\/ $a^{\pa}_{V} (0) = \dens (V)$.
\end{theorem}

\begin{proof}[Sketch of proof]
  There are at least three strategies to prove the formula for
  $a^{\pa}_{V} (y)$, each giving different insight. Here, we sketch
  a filtering argument based on a sequence of periodic
  approximants, and mention two other ones in Remark~\ref{rem:other}.
 
  One can start from the formula
  $V=\ZZ^d \setminus \bigcup_{p\in\cP}p \ZZ^d$ and realise that
  $V = \lim_{n\to\infty} V_n$ with
  $V_n \defeq \ZZ^d \setminus \bigcup_{p\in\cP_n} p \ZZ^d$ in the
  local topology, where $\cP_n$ denotes the set of the first $n$
  primes. Here, one has $V_i \supset V_{i+1}$ for all $i\in\NN$, so
  $V \subseteq \bigcap_{i\in\NN} V_i$, where we actually get equality
  from a tail estimate (via the convergence of $\zeta(2)$ from a
  Weierstrass $M$-test) together with convergence in the local
  topology, so $V= \bigcap_{i\in\NN} V_i$. Here, each $V_i$ is a
  \emph{periodic} point set with well-defined FB coefficients.  The
  latter converge as $n\to\infty$, and the claimed formula (including
  the correct density) follows from a standard inclusion-exclusion
  argument that is implicit in \cite{BMP,BH}. This establishes the
  limit-periodic structure of $V$ and how it can be viewed as the
  limit of a sequence of lattice-periodic point sets.
\end{proof}

\begin{remark}\label{rem:other}
  Theorem~\ref{thm:vp-spec} can also be derived from explicit
  convergence arguments, similar to the original proofs in
  \cite{BMP}. This way, one sees that the FB coefficients, which are
  volume-averaged exponential sums, exist and emerge from an average
  over centred balls of increasing radius, thus aligning nicely with
  the way how natural patch frequencies (and hence also the Mirsky
  measure) are defined. This shows that and how using the same
  averaging process for all quantities is relevant.
   
  Alternatively, via the approach of \cite{RS}, one can employ the
  nature of $V$ as a weak model set of maximal density \cite{BHS} in
  conjunction with the uniform distribution \cite{Moody} of the
  $\star$-image of $V$ in the window within $H$ from
  \eqref{eq:vis-cps}. This allows to calculate $a^{\pa}_{V} (y)$ via the
  Fourier transform of $\mathbf{1}^{\pa}_{W}$, which reflects the ergodic
  properties of $V$ as represented in internal space.  \exend
\end{remark} 

Two comments are in order. First, in this representation, all
eigenfunctions take only real values on $V\!$. This simply reflects
the mirror symmetry of $V\!$, which turns the FB coefficients into a
volume-averaged cosine sum. Then, the normalisation of the
eigenfunctions is not the standard (unimodular) one, but very natural
from the harmonic analysis point of view. In fact, the diffraction
measure \cite{BMP} of $V$ is the positive, pure-point measure
\[
    \widehat{\gamma} \, = \sum_{y\in L^{\circledast}} 
    \lvert a^{\pa}_{V} (y) \rvert^2 \, \delta^{\pa}_y \ts ,
\]
which gives the intensities of the peaks as the squares (of the
absolute values) of the FB coefficients, which is also known as the
phase consistency of the system \cite{LS}, while the supporting set of
$\ts\widehat{\gamma}\ts$ is the FB spectrum, which agrees with the
dynamical point spectrum in our case,\footnote{More generally, the
  dynamical point spectrum is the smallest group that contains the FB
  spectrum.} where we do not take the closure as otherwise done in
functional analysis.

Let us also take a brief look at the related MTDS
$(\YY^{\pa}_{\nts V},\RR^d,\mu^{\pa}_{\max})$, where
$\mu^{\pa}_{\max}$ is the measure of maximal entropy. It reflects the
hereditary nature of $\YY^{\pa}_{\nts V}$, and is unique. It comes
from the product of $\mu_{_\mathrm{M}}$ with the Bernoulli measure
induced by the fair coin toss \cite{Peck,KulagaLem,KLW}.  The new
measure emerges from considering the product system
$\XX^{\pa}_{V} \times \{0,1\}^{\ZZ^d}$ with $\mu_{_\mathrm{M}}$ on the
first and the Bernoulli measure on the second factor. Then, since
$\XX^{\pa}_{V}$ is hereditary, pointwise multiplication is well
defined and maps any pair $(x,z)$ to $xz$, which is always in
$\XX^{\pa}_{V}$.

This new MTDS has mixed spectrum, as a result of \cite{BLR,HR}, with a
pure-point part that is closely related to the pure-point spectrum
discussed above and an absolutely continuous part that corresponds to
countably many copies of Lebesgue measure.  Let us consider the
pure-point part constructively, which is possible because $V$ can be
turned into a generic element of $\ts\YY^{\pa}_{\nts V}$ relative to
$\mu^{\pa}_{\max}$ by a standard \emph{Bernoulli thinning}. This
refers to knocking out each point $x\in V$ randomly and independently
with probability $\frac{1}{2}$, and almost every realisation of this
process will lead to a generic element for $\mu^{\pa}_{\max}$.
Clearly, this Bernoulli thinning also works for any other
$\mu_{_\mathrm{M}}$-generic element of $\YY_V$, not just for
$V\!$. Consequently, the set of elements of $\YY^{\pa}_{V}$ we reach
this way has full measure with respect to $\mu_{\max}$.

More generally, let $p\in (0,1)$ and let
$(\xi^{\pa}_{x})^{\pa}_{x\in V}$ be a family of i.i.d.~random
variables with values in $\{ 0,1 \}$ and $\PP (\xi^{\pa}_{x} =1) =
p$. Now, let the set
$\widetilde{V} = \{ x \in V \! : \xi^{\pa}_{x} = 1 \}$ be a
realisation of the corresponding thinning, which is almost surely
generic for the pushforward of the product measure of
$\mu_{_\mathrm{M}}$ with the Bernoulli measure for $p$, where
$p=\frac{1}{2}$ corresponds to $\mu^{\pa}_{\max}$. What is more, such
a realisation almost surely has (natural) density
$\dens \bigl(\widetilde{V}\bigr) =p\, \dens (V) $, both claims
following by an application of the \emph{strong law of large numbers}
(SLLN); see the Appendix for further details.  It is now possible to
determine the FB coefficients of $\widetilde{V}\!$, where we get
\begin{equation}\label{eq:FB-def}
    a_{\widetilde{V}} (y) \, = \lim_{r\to\infty}
        \myfrac{1}{\vol ( B_{r} (0) )}
        \sum_{x\in \widetilde{V}_{r}}
        \ee^{-2\pi\ii xy} \, =   \lim_{r\to\infty}
        \myfrac{1}{\vol ( B_{r} (0) )}
        \sum_{x\in V_{r}}
        \xi^{\pa}_{x} \, \ee^{-2\pi\ii xy} .
\end{equation}

The complex random variables
$Z^{\pa}_{x} \defeq \xi^{\pa}_{x} \ee^{-2\pi\ii xy}$ are independent,
but no longer identically distributed. With values in
$\{0 \} \nts \cup \SSS^{1}$, they clearly have bounded
variance. Indeed, one has
$\EE (Z^{\pa}_x) = \EE(\xi^{\pa}_{x}) \ts \ee^{-2\pi\ii xy} = p \,
\ee^{-2\pi\ii xy}$ and
\[
    \mathbb{V} (Z^{\pa}_{x}) \, = \, 
       \EE \bigl( \lvert Z_x - \EE (Z_x) \rvert^2 \bigr)  \, = \,
       \EE \bigl( \lvert Z_x \rvert^2\bigr) - \big\lvert\EE (Z_x)
       \big\rvert^2  \, = \, p \ts (1-p) \ts .
\]
Consequently, we can apply Kolmogorov's version of the SLLN, compare
\cite[Thm.~14.5]{Bauer}, so we may replace each term
$\xi^{\pa}_{x} \ee^{-2\pi\ii xy}$ by its expectation, and almost surely
still get the same limit. For any fixed $y\in\RR^d$, this almost
surely gives
\begin{equation}\label{eq:thinned-functions}
   a_{\widetilde{V}} (y) \, = \, p \cdot \nts a^{\pa}_{V} (y) \ts .
\end{equation}
Since $V$ can be replaced by any $\mu_{_\mathrm{M}}$-generic element,
this procedure works for a subset of the hull that has full measure
for $\mu^{\pa}_{\max}$.  As explained earlier, a non-zero
$a_{\widetilde{V}} (y)$ induces an eigenfunction, which can happen for
at most countably many $y\in\RR^d$. We thus see that, almost surely,
Eq.~\eqref{eq:thinned-functions} holds for all $y\in L^{\circledast}$
together with $a_{\widetilde{V}} (y) =0$ for all other $y$. This is to
say that the point spectrum of
$(\YY^{\pa}_{\nts V}, \RR^d, \mu_{_\mathrm{M}})$ agrees with that of
$(\YY^{\pa}_{\nts V}, \RR^d, \mu^{\pa}_{\max})$, with the
eigenfunctions emerging from the FB coefficients in both cases (the
prefactor $p$ only changes the normalisation). The difference is that
these eigenfunctions now only span a subspace of
$L^2 (\YY^{\pa}_{V}, \mu^{\pa}_{\max})$.  Its complement corresponds
to the absolutely continuous part of the spectrum, with countably many
copies of Lebesgue measure as spectral measures. The above
considerations can be summarised as follows.

\begin{coro}\label{coro:mixed}
  The MTDS\/ $(\YY^{\pa}_{\nts V}, \RR^d, \mu^{\pa}_{\max})$ has mixed
  spectrum, with a pure-point and an absolutely continuous part. The
  former, in additive notation, is again given by\/ $L^{\circledast}$
  from Remark~$\ref{rem:dual-alternative}$, with a spectral measure\/
  $\delta_y$ for every\/ $y\in L^{\circledast}$.
  
  As in Theorem~$\ref{thm:vp-spec}$, the corresponding eigenfunctions
  are induced by the FB coefficients on a subset of\/ $\YY_V$ of full
  measure, but only span part of the Hilbert space\/
  $L^{2} ( \YY^{\pa}_{\nts V}, \mu^{\pa}_{\max})$.  They are augmented
  by countably many copies of Lebesgue measure as spectral measures
  for the absolutely continuous part.  \qed
\end{coro}

To go one step further, we can employ the Bernoullisation method from
\cite[Sec.~11.2]{TAO} to also calculate the almost sure diffraction
measure of $( \YY^{\pa}_{\nts V}, \mu^{\pa}_{\max})$, which gives
\[
    \widehat{\gamma}^{\pa}_{\max} \, = \, \myfrac{1}{4} \widehat{\gamma}
    + \myfrac{1}{4} \dens (V) \lambda_{_\mathrm{L}} \ts ,
\]
where $\lambda_{_\mathrm{L}}$ denotes Lebesgue measure.  This,
together with \eqref{eq:thinned-functions} shows constructively that
the new intensities of the peaks in the pure-point part are given by
the absolute value squares of the new FB coefficients. Thus, the
pure-point part still satisfies phase consistency.
  
Two comments are in order here. First, it is clear why working with
the Mirsky measure is both natural and sufficient, which is what we
will do throughout the paper. Second, in our special situation, there
is a deeper picture behind this, which revolves around uniformly
distributed sequences. This will be explained in more detail in the
Appendix.

\section{Power-free integers and some of their 
generalisations}\label{sec:power-free}

The simplest one-dimensional analogue of the visible lattice points is
the set of square-free integers. This set contains all integers that
are not divisible by the square of any (rational) prime, and thus
constitutes a special case of the class of $\cB$-free integers, with
$\cB = \{ p^2 : p \in \cP \}$. Square-free integers have been studied
for a long time, see \cite{BMP} and references therein for early
results and \cite{TAO,Peck,RS} for more recent treatments, while we
refer to \cite{Abda,KKL,DKKL,KRS,CV} and references therein for recent
results on $\cB$-free systems, and \cite{Mentzen,BaMan} for some
results on the symmetries.

A simple extension is given by
\[
  V_{\bs{\kappa}} \, \defeq \, \{ x \in \ZZ : x \text{ not divisible by }
  p^{\kappa_p} \text{ for any } p \in \cP \} \ts ,
\]
where $\bs{\kappa} = (\kappa_p)^{\pa}_{p\in\cP}$ is a sequence of natural
numbers with some constraints. In particular, $\kappa_p =1$ is only
allowed for finitely many $p\in\cP$, or in such a way that
\begin{equation}\label{eq:sum-restriction}
   \sum_{p \, : \ts\ts \kappa_p =1}p^{-1} \, < \, \infty \ts .    
\end{equation}
Further, we can also admit $\kappa_p = \infty$, which just means that
those primes are not in $\cB$ and thus give no condition. When only
finitely many $\kappa_p$ are finite, $V_{\bs{\kappa}}$ is a periodic
subset of $\ZZ$. In general, when \eqref{eq:sum-restriction} holds,
the natural density (via averaging over intervals $[-n,n]$ as
$n\to\infty$) of $V_{\bs{\kappa}}$ exists and is given by the
generally infinite, but absolutely convergent, product
\[
  \dens (V_{\bs{\kappa}}) \, = \prod_{p\in\cP} \bigl( 1-p^{-\kappa_p}
  \bigr) \, = \nts \prod_{p\in\cP'} \bigl( 1 - p^{-\kappa_p}\bigr) ,
\]
with $\cP' \defeq \{ p \in \cP : \kappa_p < \infty \}$, as follows
from a simple inclusion-exclusion argument via the periodic covering
sets that are obtained from taking the first $N$ primes into account
and then letting $N\to\infty$, in conjunction with an appropriate tail
estimate; see Section~\ref{sec:extend} below. When 
$\kappa_p = \kappa$ with $\kappa \geqslant 2$ for all
$p\in\cP$, we simply write $V_{\kappa}$, and one gets
$\dens (V_{\kappa}) = 1/\zeta (\kappa)$.

For the general case, a CPS for $V_{\bs{\kappa}}$ can be constructed as
\begin{equation}\label{eq:int-cps}
\renewcommand{\arraystretch}{1.2}\begin{array}{r@{}ccccc@{}l}
   & \RR & \xleftarrow{\,\;\;\pi\;\;\,} & \RR \times H & 
        \xrightarrow{\;\pi^{\pa}_{\mathrm{int}\;}} & H & \\
   & \cup & & \cup & & \cup & \hspace*{-1ex} 
   \raisebox{1pt}{\text{\footnotesize dense}} \\
   & \ZZ & \xleftarrow{\; 1-1 \;} & \cL & 
        \xrightarrow{\; \hphantom{1-1} \;} & \iota (\ZZ) & \\
   & \| & & & & \| & \\
   & L & \multicolumn{3}{c}{\xrightarrow{\qquad\qquad\;\,\star
       \,\;\qquad\qquad}} 
       &  {L_{}}^{\star\nts}  & \\
\end{array}\renewcommand{\arraystretch}{1}
\end{equation}
where $H=\bigotimes_{p\in\cP'} \ZZ/p^{\kappa_p} \ZZ$ is a compact
Abelian group and $\pi$ and $\pi^{\pa}_{\mathrm{int}}$ are the canonical
projections to $\RR$ and $H$, respectively.  Further, $\cL$ is a
lattice in $\RR{\ts\times\ts} H$, where
$\iota : \ZZ \longrightarrow H$ is the mapping defined by
$n \mapsto n^{\star} \defeq ( n \bmod p^{\kappa_p} )^{\pa}_{p\in \cP'}$
and $\cL$ is the diagonal embedding of $\ZZ$, so
$\cL = \{ (n, n^{\star}) : n \in \ZZ \}$. In other words, this CPS is
a simple variant of \eqref{eq:vis-cps}, with the adjusted triple
$(\RR, H, \cL)$. Then, we get
\[
  V_{\bs{\kappa}} \, = \, \{ n \in \ZZ : n^{\star} \in W \}
  \qquad \text{with } \; W \, = \Bigtimes\displaylimits_{p\in \cP'}
  \bigl(  ( \ZZ/p^{\kappa_p} \ZZ ) \setminus \{ 0 \} \bigr) .
\]
If $H$ is equipped with its normalised Haar measure
$\nh$, the \emph{window} $W$ has measure
\[
    \nh (W) \, = \prod_{p\in\cP'} 
       \frac{p^{\kappa_p} - 1}{p^{\kappa_p}}
    \, = \, \dens (V_{\bs{\kappa}}) \ts ,
\]
which shows the analogy to our guiding example. 

By a minor extension of the arguments used in \cite{BMP} for the case
of constant $\kappa_p$, one can derive the diffraction measure of the
Dirac comb
$\delta^{\pa}_{V_{\bs{\kappa}}} = \sum_{x\in V_{\bs{\kappa}}} \delta_x$
as
\[
  \widehat{\gamma^{\pa}_{V_{\bs{\kappa}}}} \, = \sum_{k\in L^{\circledast}}
  \lvert a(k) \rvert^{2} \ts \delta^{\pa}_{k}
\]
with the FB coefficients $a(k)$. The supporting set, the FB spectrum,
is
\begin{equation}\label{eq:kappa-spec}
  L^{\circledast} \, = \, \Bigl\{ k \in \QQ : \begin{array}{c} 
    \den (k) \text{ is $(\kappa_p {+} 1)$-free for all } p \in \cP' \\
    \text{and not divisible by any } q \in \cP \nts\setminus\nts \cP' 
    \end{array} \Bigr\}  \, = \, \sum_{p\in\cP'} \, p^{-\kappa_p} \ZZ \ts ,
\end{equation}
which is a subgroup of $\ts\QQ$, hence also torsion-free. It is again
calculated by the method from
Remark~\ref{rem:dual-alternative}. Further, analogously to
Theorem~\ref{thm:vp-spec}, the FB coefficients for
$k\in L^{\circledast}$ are given by
\begin{equation}\label{eq:FB-kappa-free}
    a(k) \, = \dens (V_{\bs{\kappa}}) \prod_{p \mid \den (k)}
    \frac{1}{1-p^{\kappa_p}} \ts ,
\end{equation}
while they satisfy $a (k)=0$ for all other $k$. Note that $a (k)$ is
always real, which is clear from the reflection symmetry
$V_{\bs{\kappa}} = - V_{\bs{\kappa}}$. The group $L^{\circledast}$ is
also the dynamical spectrum (in additive notation) of the MTDS
$(\YY_{\bs{\kappa}}, \RR, \mu_{_\mathrm{M}})$ with $\YY_{\bs{\kappa}}$ 
being the continuous hull of $V_{\bs{\kappa}}$ and $\mu_{_\mathrm{M}}$ 
the Mirsky measure.  Restricting to
$\ZZ$-action, one gets the spectrum $L^{\circledast}\nts \cap \TT$
with $\TT $ as above, in complete analogy to
Remark~\ref{rem:dual-discrete}.  Under addition modulo $1$, this is
the group
\[
     L^{\circledast}\! / \ZZ \, = \, \bigoplus_{p\in \cP'}
     p^{-\kappa_{p}} \ZZ / \ZZ \ts , \quad\text{with}\quad
     p^{-\kappa_{p}} \ZZ / \ZZ \simeq C_{p^{\kappa_p}} \ts .
\]
Conversely, when starting from this
discrete group action, the extension to the $\RR$-action emerges via a
standard suspension with a constant roof function; compare
\cite[Ch.~3.4]{VO}.

Again, the dynamical spectrum can also be extracted from the dual CPS,
which is a modification of \eqref{eq:vis-cps-dual}, now with $d=1$ and
adjusted entries for $H$ and $\cL^0 \nts$. The latter is
\[
   \cL^{0} \, = \, \{ (y,y^{\star}) : y \in L^{\circledast} \}  
\]
with $ \star \colon L^{\circledast} \xrightarrow{\quad} \widehat{H} =
\bigoplus_{p\in\cP'} \ZZ/p^{\kappa_{p}}\ZZ$
being defined as follows. First, consider the terms in
$\prod_{p \mid \den (u)} p^{\kappa_p}$ and then set
\[
     u^{\star} \, = \,  \bigl( - \lcm (\den (u), p^{\kappa_p}) \ts u
        \bmod p^{\kappa_p} \bigr)_{p\in\cP'} \ts ,
\]
which is a slight modification of our $\star$-map for the visible
lattice points. Once again, the entry for any $p \nmid \den (u)$
vanishes, so that $\ZZ$ is the kernel of the group homomorphism
$\star$.  All of this follows from calculations analogous to those of
our guiding example, which we leave to the interested reader; see also
\cite{RS}. Note that this formulation is also correct when $\cP'$ is a
finite set, in which case $V_{\bs{\kappa}}$ is periodic. The resulting
expressions can then also be seen via a simple inclusion-exclusion
argument, which is the basis of the argument from
Remark~\ref{rem:dual-alternative}.

Let us summarise the above derivations and recollections as follows. 

\begin{theorem}\label{thm:k-free-integers}
  Let\/ $\bs{\kappa} = (\kappa^{\pa}_{p})^{\pa}_{p\in\cP}$ with\/
  $\kappa^{\pa}_{p} \in \NN \cup \{ \infty \}$ be such that
  Eq.~\eqref{eq:sum-restriction} holds, and let\/ $V_{\bs{\kappa}}$ be
  the corresponding set of\/ $\bs{\kappa}$-free integers. Then, the
  diffraction measure of\/ $V_{\bs{\kappa}}$ is pure point, and given
  by
\[ 
     \widehat{\gamma^{\pa}_{\bs{\kappa}}} \, = \sum_{k \ts \in L^{\circledast}}
     \lvert a(k) \rvert^2 \ts \delta^{\pa}_{k}
\]   
with the FB spectrum\/ $L^{\circledast}$ from\/ \eqref{eq:kappa-spec}
and the FB coefficients\/ $a(k)$ from\/ \eqref{eq:FB-kappa-free}.
   
Further, the induced dynamical system\/
$(\XX_{\bs{\kappa}}, \ZZ, \mu_{_\mathrm{M} } )$ has pure-point
spectrum with trivial topological point spectrum. In additive
notation, the spectrum is given by the Abelian group\/
$L^{\circledast} \! / \ZZ = \bigoplus_{p\in\cP'} \ts p^{-\kappa_p}\ZZ
/ \ZZ $, with\/ $\cP' = \{ p\in\cP : \kappa_{p} < \infty \}$.  \qed
\end{theorem}

At this point, our previous comments on $\mu_{_\mathrm{M}}$ versus
$\mu^{\pa}_{\max}$ apply in complete analogy. In particular, we could
repeat Corollary~\ref{coro:mixed} for this systems as well, which we
omit.

\begin{remark}
  With hindsight, we could have allowed for an analogous
  generalisation in the set of visible lattice points as well, by
  considering $\ZZ^d \setminus \bigcup_{p\in\cP} p^{\kappa_p} \ZZ^d$,
  where $\kappa_p \in \NN \cup \{\infty\}$ is now allowed for all $p$,
  with the obvious meaning for $\kappa_p=\infty$.  The entire analysis
  extends accordingly, without providing additional insight, wherefore
  we skip further details.  \exend
\end{remark}

\section{A more general setting with \texorpdfstring{$\cB$}{}-free
  lattice systems}\label{sec:extend}

Though we are particularly interested in the number-theoretic setting,
a more general approach would employ lattices. Indeed, consider
$\RR^d$, with some fixed norm $\| . \|$, and fix some lattice $\vG$ in
it. The choice $\vG = \ZZ^d$ will then correspond to the situation of
symbolic dynamics.  Now, select a collection of lattices
$\cB = \{ \vG_i : i \in \NN \}$ with the following properties.
\begin{itemize}\itemsep=2pt
\item[(B1)] One has $\vG_i \subset \vG$ with $1 < [\vG : \vG_i] 
    \leqslant [\vG : \vG_{i+1}]$ for all $i\in \NN$.
\item[(B2)] One has $(\vG_1 \cap \ldots \cap \vG_m) 
    + \vG_{m+1} = \vG$ for all $m \in \NN $.
\item[(B3)] One has $\sum_{i\in\NN} \lambda (\vG_i)^{-d} < \infty$,
    where $\lambda (\vG) = \min_{x\in \vG\ts\setminus \nts\{ 0 \}} \| x \|$.
\end{itemize}
Note that (B2) implies $\vG_i + \vG_j = \vG$ for all $i \ne j$. The latter
is equivalent to (B2) when $d=1$, but not for $d>1$. There, (B2) is a
stronger condition, which is equivalent to pairwise coprimeness in
conjunction with the condition that the natural projection
of $\vG$ into $H = \bigotimes_{i\in\NN} \vG / \vG_i$ is dense. The 
latter is true if and only if $\vG$ projects surjectively to 
$\bigotimes_{i=1}^{m} \vG/\vG_i$ for all $m\in \NN$. In fact, 
(B2) ensures the lattice version of the CRT, which 
we need in our density arguments.

Now, we consider $V = \vG \setminus \bigcup_{i\in\NN} \vG_i$, which
generalises our previous examples. Also all later ones will be of this
type. These conditions are one possible generalisation of Erd\H{o}s
$\cB$-free systems \cite{KKL,KLW} from one to higher dimensions
\cite{BBHLN}. We note that (B3) is sufficient for what we need, but
not necessary for $d>1$; slightly weaker conditions will still work,
but are not pursued here.  Let us mention here that the above setting
can be extended to cover more general sieves. For instance, one could
be interested in removing not just $\vG_i$, but also some cosets of it
in $\vG$, which we leave for future work.

Set $V_n = \vG \setminus \bigcup_{i=1}^{n} \vG_i$, which is still a
periodic point set, with $\bigcap_{i=1}^{n} \vG_i$ as its lattice of
periods. Note that $V^{\pa}_{0} \defeq \vG$ is a limiting case, with
$\vG$ as its lattice of periods. The density of any $V_n$ exists
uniformly.  With $[n]\defeq \{ 1, \ldots , n \}$ for $n\in\NN$ and
$[0]\defeq \varnothing$, one has
\[
   \dens (V_n) \, = \sum_{I\subseteq [n]} (-1)^{\lvert I \rvert} 
        \ts \dens (\vG^{\pa}_I) 
\]
by a standard inclusion-exclusion argument, where
$\vG^{\pa}_I = \bigcap_{i\in I} \vG_i$ with
$\vG^{\pa}_{\varnothing} \defeq \vG$ and $\lvert I \rvert$ denotes the
cardinality of the finite set $I$. Due to condition (B2), we simply
have
\[
    \dens (\vG^{\pa}_{I}) \, = \, \dens (\vG) 
        \prod_{i\in I} [\vG : \vG_i]^{-1} \ts ,
\]     
which then gives, with $N_i \defeq [\vG : \vG_i]$,
\begin{equation}\label{eq:dens-Vn}
  \dens (V_n) \, = \dens (\vG) \prod_{i=1}^{n} 
  \Bigl( 1 - \myfrac{1}{N_{i}} \Bigr) .
\end{equation}
As $n\to\infty$, the product converges as a consequence of (B3),
because $N_i \geqslant \lambda (\vG_{i})^d$.

In contrast, the density of $V$ will generally not be uniform, and is
more complicated.  Since $V\subset V_n$ for every $n\in \NN$, we get
$\overline{\dens} (V) \leqslant \dens (V_n)$ for all $n\in\NN$, where
the upper density is defined with respect to any given
 F{\o}lner averaging sequence.  This implies
\[
    \overline{\dens} (V) \, \leqslant \, \inf_{n\in\NN} \dens (V_n)
    \, =  \lim_{n\to\infty} \dens (V_n) \, =  \, \dens (\vG)
    \prod_{i=1}^{\infty} \Bigl( 1 - \myfrac{1}{N_i } \Bigr) ,
\]
where the first equality follows from $V_n \supset V_{n+1}$ for
all $n\in\NN$, and the second from \eqref{eq:dens-Vn}.

In general, the corresponding lower density could differ, but our
special conditions are chosen so that this does not happen for the
natural (tied) density with respect to centred balls. So, let
$B_r = B_r (0)$ be the closed ball of radius $r$ around $0$ and
consider $\cA = \{ B_r : r > 0 \}$. For any point set
$S\subset \RR^d$, we now use
\[
   \overline{\dens}_{\cA} (S) \, \defeq  \limsup_{r\to\infty} 
   \myfrac{\lvert S\cap B_r \rvert}{\vol (B_r)} 
   \quad \text{and} \quad \underline{\dens}_{\ts\cA} (S) \, \defeq 
   \liminf_{r\to\infty}  \myfrac{\lvert S\cap B_r \rvert}{\vol (B_r)} 
\]
and say that the density of $S$ with respect to $\cA$ exists if
$\overline{\dens}_{\cA} (S) = \underline{\dens}_{\ts\cA} (S)$. From
the above, we get
\[
  \overline{\dens}_{\cA} (V) \, \leqslant \, \overline{\dens}_{\cA}
  (V_n) \, = \, \dens_{\cA} (V_n) \, = \, \dens (V_n)
\]
for every $n\in\NN$, where $\dens (V_n)$ refers to the uniform density
of $V_n$ from above. Since $V\subset V_n$, we have the disjoint union
$V_n = V \ts\dot{\cup}\ts \vD_n$ with
$\vD_n = V_n \!\setminus\nts\nts V\nts$.

Now, let $(A_i)^{\pa}_{i\in\NN}$ with $A_i = B_{r_i}$ be a sequence of
growing balls (with unbounded radii) such that
$\underline{\dens}_{\cA} (V) = \lim_{i\to\infty} \frac{\lvert V \cap
  A_i \rvert}{\vol (A_i)}$, which must exist by standard arguments.
Observing that
$V_n \cap A_i = (V \cap A_i) \, \dot{\cup} \, (\vD_n \cap A_i)$, we
get the estimate
\[
  \dens (V_n) \, = \ts \lim_{i\to\infty} \myfrac{\lvert V_n \cap A_i
    \rvert} {\vol (A_i)} \, \leqslant \, \underline{\dens}_{\ts \cA}
  (V) + \overline{\dens}_{\cA} (\vD_n)
\]
and hence
$\underline{\dens}_{\ts \cA} (V) \geqslant \dens (V_n) -
\overline{\dens}_{\cA} (\vD_n)$ for all $n\in\NN$. Now, the upper
density of $\vD_n$ gets arbitrarily small with growing $n$, as a
consequence of condition (B3), so lower and upper density of $V$ with
respect to $\cA$ must agree, and we have
\[
     \dens_{\cA} (V) \, = \lim_{n\to\infty} \dens (V_n) \ts .
\]
Invoking the results from \cite{BHS}, this has the following consequence.

\begin{prop}\label{prop:gen-latt}
  Let\/ $\vG\subset \RR^d$ be a fixed lattice, with dual\/ $\vG^{*}$, and 
  let\/ $\cB = \{ \vG_i : i\in\NN \}$ be a family of lattices in\/ $\RR^d$
  subject to conditions\/ \textnormal{(B1) -- (B3)}. Then,
  $V = \vG \setminus \bigcup_{i\in\NN} \vG_i$ is a weak model set of
  maximal natural density. As such, it has pure-point diffraction,
  with FB spectrum\/ $L^{\circledast} = \sum_{i\in\NN}
  \vG_{i}^{*}$, where  $L^{\circledast}\! / \vG^{*} \nts = \bigoplus_{i\in\NN}
  \vG^{*}_{i}/\vG^{*}$ is the induced subgroup of\/
  $\TT^d = \RR^d/\vG^{*}\!$.
\end{prop}

\begin{proof}
  The CPS is $(\RR^d, H, \cL)$ with the compact Abelian group
  $H = \bigotimes_{i\in\NN} \vG/\vG_i$ and the lattice $\cL$ that
  emerges as the diagonal embedding of $\vG$ into $\RR^d \times H$.
  Concretely, one has $\cL = \{ (t, t^{\star}) : t \in \vG \}$ with
  $t^{\star} = (t \bmod \vG_i)^{\pa}_{i\in\NN}$.
  
  Then, $V = \{ t \in \vG : t^{\star} \in W \}$ is a model set with
  the window
  $W = \bigtimes_{\! i\in\NN} \bigl( H_i \setminus \{0\} \bigr)$, with
  $H_i = \vG / \vG_{i}$.  Here, $V$ is a weak model set because $W$
  has no interior. Still, the volume of $W$ in the normalised Haar
  measure of $H$ is
\[
  \vol (W) \, = \, \prod_{i\in\NN} \Bigl( 1 - \myfrac{1}{N_{i} }\Bigr),
\]  
which is a converging product as a consequence of condition (B3).
Then, $\dens (\vG) \ts \vol (W)$ is the maximal density for the
corresponding weak model set.  Since this is the natural density of
$V$, it is maximal as claimed.
  
The pure-point nature is now a consequence of \cite{BHS}, and the FB
spectrum follows from another application of
Remark~\ref{rem:dual-alternative}, by viewing $V$ as the limit of the
sequence $(V_n)^{\pa}_{n\in\NN}$, both in the local topology and on
average. The set $V$ is then limit periodic, with the spectrum as
stated. Further, each $\vG^{*}_{i} / \vG^{*}$ is a discrete subgroup of
the fundamental domain $\TT^d = \RR^d / \vG^{*}$ under addition modulo
$\vG^{*}\!$, in obvious extension of Remark~\ref{rem:dual-discrete}.
\end{proof}

Note that each $\vG^{*}_{i}\nts /\vG^{*}$ is a finite Abelian group,
though its decomposition into cyclic groups is not possible without
further information on the lattices in $\cB$.

The corresponding TDS in the setting of symbolic dynamics is obtained
via the orbit closure $\XX^{\pa}_V$ of $1^{\pa}_{V}$ under the
translation action of $\vG$, which gives $(\XX^{\pa}_V, \vG)$. The
space $\XX^{\pa}_{V} \subset \{ 0, 1\}^{\vG}$ is compact and known as
the \emph{discrete hull} of $V$ under the action of $\vG$. Its flow
counterpart is $(\YY^{\pa}_{\nts V}, \RR^d)$, with
$\YY^{\pa}_{\nts V}$ the orbit closure under the translation action of
$\RR^d$ as the \emph{continuous hull}. When equipped with the Mirsky
measure, which is once again obtained as the patch frequency measure
derived from $V$ relative to a tied F{\o}lner averaging sequence
$\cA$, we get pure-point dynamical spectrum. It is, in additive
notation, given by $L^{\circledast}\cap \TT^d$ for
$(\XX^{\pa}_{V},\vG, \mu_{_\mathrm{M}})$ and by $L^{\circledast}$ for
$(\YY^{\pa}_{\nts V}, \RR^d, \mu_{_\mathrm{M}})$, respectively. No
eigenfunction of $(\XX^{\pa}_{V}, \vG, \mu_{_\mathrm{M}})$, except the
trivial one, admits a continuous representation. For
$(\YY^{\pa}_{V}, \RR^d, \mu_{_\mathrm{M}})$, the only continuous
eigenfunctions are the ones for the eigenvalues from $\vG^{*}\!
$. Still, starting from $V$, which is a generic set for both dynamical
systems, the eigenfunctions can be given by the FB coefficients.

To compute them, we first recall that, for any finite $I\subset \NN$,
the FB coefficient of $\vG^{\pa}_{I}$ is
\[
    a^{\pa}_{\vG^{\pa}_{\! I}} (k) \, = \lim_{r\to\infty} \myfrac{1}{\vol (B_r)}
    \sum_{x\in\vG^{\pa}_{\! I}\cap B_r} \! \ee^{-2 \pi \ii kx} \, = \,
    \begin{cases} \dens (\vG^{\pa}_{I}) , & \text{if } k\in\vG^{*}_{I}, \\
     0 , & \text{otherwise}. \end{cases}
\]
This gives the FB coefficients of $V_n$ via inclusion-exclusion as
$ a^{\pa}_{V_n} (k) = \sum_{I \subseteq [n]} (-1)^{\lvert I \rvert}
a^{\pa}_{\vG^{\pa}_{\! I}} (k) $, which vanishes for all
$k \notin \vG^{*}_{[n]}$.  This formula also holds in the limit as
$n\to\infty$, where one gets non-zero results only for
$k\in L^{\circledast}$. If $k\in L^{\circledast}$, one can define
\[
  D (k) \, = \, \{ i\in\NN : k \in \vG^{*}_{i} \setminus \vG^{*} \}
  \ts ,
\]
which is a finite set and takes the role of the denominator of a 
rational vector in our guiding examples, due to the property (B2). 
Then, for all  $k\in L^{\circledast}$, one finds
\begin{equation}\label{eq:gen-FB}
    a^{\pa}_{V} (k) \, = \, \dens^{\pa}_{\cA} (V) 
    \prod_{i\in D(k)} \myfrac{1}{1 - N_{i} } 
\end{equation}
together with $a^{\pa}_{V} (k) = 0$ for all other $k$.  The verification
of this formula uses the very same inclusion-exclusion calculation
that underlies the FB coefficient in Theorem~\ref{thm:vp-spec}.

The eigenfunction property follows once again from the relation
$a^{\pa}_{t+V} (k) = \ee^{-2 \pi \ii k t} a^{\pa}_{V} (k)$ in conjunction
with the fact that $V$ is a generic element for both hulls.  Putting
everything together, we have the following result.

\begin{theorem}\label{thm:gen-latt}
  Under the assumptions of Proposition~$\ref{prop:gen-latt}$, the
  set\/ $V$ induces the measure-theoretic dynamical system\/
  $(\XX^{\pa}_{V}, \vG, \mu_{_\mathrm{M}})$ and its suspension to the
  flow\/ $(\YY^{\pa}_{V}, \RR^d, \mu_{_\mathrm{M}})$.  Both have
  pure-point dynamical spectrum, given by\/
  $L^{\circledast} \! / \vG^{*} = \bigoplus_{i\in\NN}
  \vG^{*}_{i}/\vG^{*}$ for the former and by\/
  $L^{\circledast}=\sum_{i\in\NN} \vG^{*}_{i}$ for the latter.
   
  The eigenfunctions are defined via \eqref{eq:gen-FB} on an orbit of
  a generic element, and similarly for all generic elements of maximal
  density, which have full measure. However, they do not have a
  continuous representative unless\/ $k=0$ respectively\/
  $k\in\vG^{*}\!$.  \qed
\end{theorem}

\begin{remark}\label{rem:max-entropy}
  As in our guiding example, one can replace $\mu_{_\mathrm{M}}$ by
  the pushforward of its product with the Bernoulli measure of the
  fair coin toss to obtain $\mu_{\max}$, the measure of maximal
  entropy. The corresponding Bernoulli thinning process applied to $V$
  will produce (almost surely) a realisation that is generic for
  $\mu^{\pa}_{\max}$, and our argument with Kolmogorov's version of
  the SLLN extends to this case, so that the eigenfunctions of both
  $(\XX^{\pa}_{V}, \vG, \mu_{_\mathrm{M}})$ and
  $(\XX^{\pa}_{V}, \vG, \mu_{\max})$ derive from the FB coefficients.
  For $\mu_{\max}$, however, they span only a subspace of
  $L^2 (\XX^{\pa}_{V}, \mu^{\pa}_{\max})$, and similarly for the
  corresponding flows under $\RR^d \nts$. We leave further details to
  the reader.  \exend
\end{remark}

While these results look nice and decently general, it is not possible
to compute the spectrum much further unless additional properties are
present. Fortunately, this is the case in the number-theoretic
setting, as we shall now discuss in detail for one class of examples.

\section{Quadratic number fields}\label{sec:quadratic}

Quadratic number fields come in two flavours, real and imaginary,
which require slightly different computations. After some
recollections with focus on our needs, we will thus treat them
separately; see \cite{Zagier} for general background.

Any quadratic number field is of the form $K = \QQ (\sqrt{d\ts}\,)$
for some integer $1 \ne d \in \ZZ$ that is square-free. The ring of
integers in $K$, which is a maximal order, is
\begin{equation}\label{eq:order}
  \cO^{\pa}_{K} \, = \, \begin{cases} \ZZ \bigl[
     \frac{1 + \sqrt{d\ts}}{2} \ts \bigr], &
    \text{if $d \equiv 1 \bmod 4$}, \\
      \ZZ[ \sqrt{d\ts}\,], & \text{otherwise}. \end{cases}
\end{equation}
When $d<0$, the unit group $\cO^{\times}_{K}$ is finite, and
isomorphic to $C_4$ ($d=-1$), to $C_6$ ($d=-3$), or to
$C_2 = \{ \pm 1 \}$ (in all remaining cases). In contrast,
$\cO^{\times}_{K} \simeq C_2 \times C_{\infty}$ for all real quadratic
fields, where the $C_{\infty}$ is generated by a fundamental unit.

Let us recall a useful result that gives access to the Abelian
factor groups $\cO/\fp^{\kappa}$ for a prime ideal $\fp$ and
$\kappa\in\NN$, which follows from the material in
\cite[Chs.~I.3 and I.12]{Neukirch}.

\begin{fact}\label{fact:factors}
  Let\/ $K$ be a quadratic number field, with\/ $\cO = \cO^{\pa}_{K}$
  as its ring of integers. If\/ $\fp \subseteq \cO$ is a prime ideal
  over the rational prime\/ $p$ and\/ $\kappa\in\NN$, one has
\[
    \cO\nts / \fp^{\kappa} \ts \simeq \, \begin{cases}
      C_{p^{\lceil \nts \kappa/2 \rceil}} \times
          C_{p^{\lfloor \nts \kappa/2 \rfloor}} , &
          \text{if $p$ is ramified,} \\
      C_{p^\kappa}\nts \times C_{p^\kappa} , & \text{if $p$ is inert,} \\
      C_{p^{\kappa}}, & \text{if $p$ is a splitting prime,} \end{cases}
\]
where the norm of\/ $\fp$ is\/ $p^2$ if\/ $p$ is inert and\/
$p$ in the other two cases.  \qed
\end{fact}

Let us next state a result that will allow us to verify property (B3)
in our later examples.  From now on, we shall mostly suppress the
index $K$, and simply write $\cO$ instead of $\cO^{\pa}_{K}$. Since we
shall need some simple estimates later on to verify the condition
(B3), we recall the following property for the \emph{absolute norm} of
non-zero elements in $\cO$, as defined by the index of $x\ts\cO$ in
$\cO$, so $\No (x) = \No (x \ts \cO) \defeq [\cO : x \ts \cO]$.

\begin{fact}\label{fact:norm-est}
  Let\/ $K$ be a quadratic number field, with\/ $\cO = \cO^{\pa}_{K}$ as
  its ring of integers. If\/ $\fa \subseteq \cO$ is a non-zero ideal
  and\/ $0\ne x\in \fa$, one has\/ $\No (x) \geqslant \No (\fa)$.
  Moreover, if\/ $\fa$ is a principal ideal, one has equality for the
  generating element of\/ $\fa$.
\end{fact}

\begin{proof}
  If $0 \ne x \in \fa \subseteq \cO$, we have $x\ts \cO \subseteq 
  \fa \subseteq \cO$ and thus
\[ 
  \No (x) \, = \, [ \cO : x \ts \cO ] \, = \,
  [\cO : \fa ] \, [ \fa : x \ts \cO] \, = \,
  \No (\fa) \ts [ \fa : x \ts \cO] \, \geqslant \,
  \No (\fa)
\]  
because $[\fa : x \ts \cO]$ is a positive integer. If $\fa$ is
principal, one has $\fa = y \ts \cO$ for some $y\in\cO$, which gives
$[\fa : y \ts \cO]=1$ and thus the second claim.
\end{proof}

The ring $\cO$ is a lattice in $\CC\simeq \RR^2$ when $d<0$, while the
lattice property requires a Minkowski embedding for $d>1$.  For our
discussion of power-free integers, it is thus convenient to consider
imaginary and real quadratic fields separately, as we shall now do.

\subsection{Imaginary quadratic fields}

If $K$ is an imaginary quadratic field, $\cO = \cO^{\pa}_K$ is a lattice
in $\CC \simeq \RR^2$, such as the Gaussian integers $\ZZ[\ii]$ from
$K=\QQ(\ii)$. Here, we identify $z=x+\ii y \in\CC$ with the column
vector $(x,y)^{\trans} \! \in \RR^2$ as usual.  For our calculations,
we now need a symmetric {$\QQ$-bilinear} form on $K$ that equals 
the standard Euclidean scalar product in $\RR^2 $, which is
\begin{equation}\label{eq:form-complex}
  x.y \, \defeq \, \myfrac{1}{2} (\ol{x}\ts y + x \ts \ol{y})
  \, = \, \mathrm{Re} (\ol{x} \ts y) \ts .
\end{equation}
The dual of $\cO$ then is
\[
  \cO^{*} \, = \, \{ y \in K : x.y \in \ZZ \text{ for all }
          x \in \cO \} \ts ,
\]
which is also the dual of $\cO$ in the sense of lattices in
$\RR^2\nts $, explaining the factor $\frac{1}{2}$ in
\eqref{eq:form-complex}.

Let us now write $K = \QQ (\ii \sqrt{\delta\ts}\,)$ with
$\delta \in \NN$ square-free, which turns the original condition
$d\equiv 1 \bmod 4$ into $\delta \equiv 3 \bmod 4$. Then, the
following result is a special case of a more general property; see
\cite[Prop.~III.2.4]{Neukirch}.

\begin{lemma}\label{lem:dual-order-complex}
  Let a square-free\/ $\delta\in\NN$ be fixed and consider the
  imaginary quadratic field\/ $K = \QQ(\ii \sqrt{\delta\ts}\,)$, with
  its ring of integers\/ $\cO = \cO^{\pa}_{K}$ according to
  Eq.~\eqref{eq:order} for\/ $d = -\delta$.  Then, the dual of\/ $\cO$
  with respect to the\/ $\QQ$-bilinear form \eqref{eq:form-complex} is
\[
  \cO^{*} \, = \, \myfrac{\ii}{\sqrt{\delta\ts}} \begin{cases}
    2 \ts \cO , & \text{if } \delta \equiv 3 \bmod 4 \ts , \\
    \cO , & \text{otherwise} \ts . \end{cases}
\]  
\end{lemma}

\begin{proof}
  When $\delta \not\equiv 3 \bmod 4$, we have
  $\cO = \langle 1, \ii \sqrt{\delta\ts}\,\rangle^{\pa}_{\ZZ}$. It is
  easy to check that
  $\bigl\{ 1 , \frac{\ii}{\sqrt{\delta\ts}} \bigr\}$ is the dual basis
  with respect to our bilinear form \eqref{eq:form-complex}.  We thus
  have
\[
     \cO^{*} \ts = \, \Big\langle 1 , \myfrac{\ii}{\sqrt{\delta\ts}}
     \Big\rangle_{\ZZ} \, = \, \myfrac{\ii}{\sqrt{\delta\ts}} \,
     \big\langle {-}\ii \sqrt{\delta\ts}, 1 \big\rangle^{\pa}_{\ZZ}
     \, = \, \myfrac{\ii}{\sqrt{\delta\ts}} \,
     \big\langle 1, \ii \sqrt{\delta\ts} \,\big\rangle^{\pa}_{\ZZ}
     \, = \, \myfrac{\ii}{\sqrt{\delta\ts}} \ts \cO
\]  
  for all these cases.

  Likewise, when $\delta\equiv 3 \bmod 4$, we have
  $\cO = \big\langle 1, \frac{1 + \ii \sqrt{\delta\ts}}{2} \ts
  \big\rangle_{\ZZ}$.  Here, the dual basis with respect to
  \eqref{eq:form-complex} is
  $\big\{ 1 {-} \frac{\ii}{\sqrt{\delta\ts}} , \frac{2
    \ii}{\sqrt{\delta\ts}} \big\}$, so that we get
\[
    \cO^{*}\ts = \, \Big\langle 1{-}\myfrac{\ii}{\sqrt{\delta\ts}} , 
      \myfrac{2 \ii}{\sqrt{\delta\ts}} \Big\rangle_{\ZZ} \, = \,
      \myfrac{2\ii}{\sqrt{\delta\ts}} \Big\langle 
      \myfrac{{-}1{-}\ii\sqrt{\delta\ts}}{2}, 1 \Big\rangle_{\ZZ}
      \, = \, \myfrac{2\ii}{\sqrt{\delta\ts}} \Big\langle 1,
         \myfrac{1{+}\ii\sqrt{\delta\ts}}{2} \Big\rangle_{\ZZ}
         \, = \, \myfrac{2 \ii}{\sqrt{\delta\ts}} \ts \cO \ts ,
\]  
which establishes the claim for this case.
\end{proof}

From here, one can calculate the dual of non-trivial principal ideals
via
\begin{equation}\label{eq:dual-principal-ideal}
  (z)^{*} \ts = \, (z \cO)^{*} \ts = \,
  \myfrac{1}{\ts \ol{z} \ts } \ts \cO^{*} .
\end{equation}
This allows to calculate the spectrum for all imaginary quadratic
fields with class number $1$, which are the nine \cite[p.~83]{Zagier}
fields $\QQ(\ii \sqrt{\delta\ts}\,)$ with
\[
     \delta \in  \{ 1, 2, 3, 7, 11, 19, 43, 67, 163 \} \ts .
\]
For the general case, we need the dual of an arbitrary non-trivial
ideal, $\fa$, which is given by
\begin{equation}\label{eq:dual-ideal-complex}
     \fa^{*} \ts = \, \ol{\fa}^{\ts -1} \ts \cO^{*} 
\end{equation}
where $\ol{\fa}$ denotes the complex conjugate ideal and
$\fb^{-1} \defeq \{ x\in K : x \fb \subseteq \cO \}$ is the
\emph{inverse} of $\fb$, and a fractional ideal. In particular,
$\fb^{-1} \fb = \cO$. Since $\cO$ is a Dedekind domain, all
non-trivial ideals are invertible; see \cite[Ch.~I.4]{J} for
background material. When $\fa$ and $\fb$ are coprime ideals, 
so $\fa + \fb = \cO$, one has
\[
  ( \fa \fb )^{*} \ts = \, ( \fa \cap \fb )^{*} \ts =
  \, \fa^{*} \nts + \fb^{*} \ts = \, \bigl( \ol{\fa}^{\ts -1}
  + \ol{\fb}^{\ts -1} \bigr) \cO^{*} \ts = \,
  \mathfrak{c}^{-1} \cO^{*}
\]
for a unique fractional ideal $\mathfrak{c}$. The corresponding
formula holds for the dual of multiple intersections, provided all
ideals involved are pairwise coprime, which is the ideal-theoretic
version of the CRT \cite[Thm.~3.6]{Neukirch}. In this sense, we can
always define the \emph{denominator} via the fractional ideal
$\mathfrak{c}$, compare \cite[Cor.~3.9]{Neukirch}, and then decompose
the denominator into prime ideals in the Dedekind domain $\cO$.

When viewing $\cO$ as a lattice in $\RR^2$, it is often useful to pin
down a basis matrix for it in terms of the standard Cartesian basis of
$\RR^2$. Here, for $\delta \equiv 3 \bmod 4$, respectively for all
other cases, natural choices are
\begin{equation}\label{eq:basis-imag}
  B \, = \, \myfrac{1}{2} \begin{pmatrix} 2 & 1 \\
    0 & \sqrt{\delta\ts} \end{pmatrix}
  \quad \text{and} \quad
  B \, = \, \begin{pmatrix} 1 & 0 \\ 0 & \sqrt{\delta\ts}
  \end{pmatrix} .
\end{equation}
Let us discuss two quadratic fields of widespread interest, namely
$\delta = 1$ and $\delta = 3$, before we continue with the general
case.

\subsubsection{Gaussian integers}
Since $\ZZ[\ii] = \ZZ^2$, both arithmetically and geometrically, we
begin with this case.  Also, we profit from $\ZZ[\ii]$ being a
Euclidean domain (and thus having unique prime factorisation up to
units, that is, up to $\{ 1, \ii, -1, -\ii \}$).

Relative to $\cO$, there are three kinds of rational primes, namely
$2 = - \ii (1+\ii)^2$, which is the only \emph{ramified} prime, the
primes $p \equiv 3 \bmod 4$, which are \emph{inert} and thus stay
prime in $\cO$, and the \emph{splitting} primes, $p \equiv 1 \bmod 4$,
giving complex conjugate pairs of Gaussian primes via
$p = \fp \ts \ol{\fp}$, such as $5=(2+\ii) (2-\ii)$ or
$13=(3+2\ii)(3-2\ii)$.  We thus get the Gaussian primes
\[
   \pg \, = \, \{ 1{+}\ts\ii, 3, \fp_{(5)}, \ol{\fp}_{(5)}, 7, 11,
     \fp_{(13)}, \ol{\fp}_{(13)},  \fp_{(17)}, \ol{\fp}_{(17)}, 
     19, \ldots \} \ts .
\]
Here, they are listed along increasing rational primes and some rule
for the splitting primes, such as taking a representative with
positive real and imaginary parts and bigger real part first. Note
that they are only unique up to units.

As a lattice in $\RR^2$, the ring $\ZZ[\ii]$ is self-dual, as follows
from Lemma~\ref{lem:dual-order-complex} because $\delta = 1$ and $\ii$
is a unit. For any principal ideal $(z) = z \ts \ZZ[\ii]$ with
$z\ne 0$, we get its dual from \eqref{eq:dual-principal-ideal} as
\[
  (z)^{*} \, = \, \myfrac{1}{\ts\bar{z}\ts} \ts \ZZ[\ii] \, = \,
  \myfrac{1}{\ts |z|^2} \ts (z) \ts ,
\]
which makes the calculations in this case particularly simple. Note
that all principal ideals are square lattices again.

Now, let $\bs{\kappa} = (\kappa_{\fp})^{\pa}_{\fp\in \pg}$ with
$\kappa_{\fp} \in \NN\cup\{\infty\}$, set $\pg' = \{ \fp \in \pg :
\kappa_{\fp} < \infty \}$ and define
\[
  V_{\sg ,\bs{\kappa}} \, = \, \ZZ[\ii] \, \setminus
  \!\bigcup_{\fp\in\pg'}\!  \bigl( \fp^{\kappa_{\fp}}\bigr) ,
\]
with the special case $V_{\sg, \kappa}$ for constant
$\kappa_{\fp} = \kappa$, for all $\fp\in\pg$.  For our further
analysis, we need some restriction on $\bs{\kappa}$, as indicated
earlier. If $\fa \subseteq \cO$ is an ideal, we define $\lambda (\fa)$
from the condition (B3) for $\fa$ viewed as a lattice, with respect to
the Euclidean norm. We then need
\begin{equation}\label{eq:tail-cond}
     \sum_{\fp\in\pg'} \lambda (\fp^{\kappa_{\fp}})^{-2}
      \, < \,  \infty 
\end{equation}
as our condition (B3). This can now be related to the norms of the
ideals via Fact~\ref{fact:norm-est}, then giving the slightly simpler
sufficient condition
\begin{equation}\label{eq:G-cond}
    \sum_{\fp\in\pg'} \No (\fp)^{-\kappa^{\pa}_{\fp}} \, < \, \infty \ts .
\end{equation}    
To see this, observe that the field norm of $z\in K$ is
$N (z) = z \ol{z}$ and agrees with the squared length of $z$ when
viewed as a vector in $\RR^2$. As the relation to the absolute norm is
given by $\No (z) = [ \cO : z \cO] = z \ol{z} = N (z)$, we can use
Fact~\ref{fact:norm-est} for $\lambda (\fa)$.

Let us expand on the restrictions for the choice of $\bs{\kappa} $.
When $\pg'$ is a finite set, without any further restriction on the
exponents $\kappa^{\pa}_{\fp}$, the point set $V^{\pa}_{\sg, \bs{\kappa}}$
is lattice periodic, and the FB spectrum simply is the dual of the
lattice of periods.  Let us thus consider the case that $\pg'$ is an
infinite set.  Since only finitely many rational primes are ramified
over $K$, there are no restrictions on the corresponding exponents.
When $p$ is inert, the shortest non-zero element in $p\cO$ has length
$\lambda (p \cO) = \sqrt{ \No (p \cO)} = p$, so that we only need
$\kappa^{\pa}_{p} \geqslant 1$ here. Finally, when $p = \fp \ol{\fp}$
splits, we get $\lambda (\fp) = \lambda (\ol{\fp}) = p$, which implies
that the corresponding exponents need to be $\geqslant 2$ for almost
all of them, or for sufficiently many so that \eqref{eq:G-cond} and
thus \eqref{eq:tail-cond} holds.

The principal ideal $(\fp^{m} )$ is a square sublattice of $\ZZ[\ii]$
of index
\[
  \bigl[ \ZZ[\ii] : (\fp^m) \bigr] \, = \, \bigl[ \ZZ[\ii] : \fp^m
  \ZZ[\ii] \bigr] \, = \, \No (\fp)^m ,
\]
from which one derives the natural density of $V_{\sg,\bs{\kappa}}$ as
\[
    \dens ( V_{\sg, \bs{\kappa}} ) \, =  \ts \prod_{\fp\in\pg'} 
     \bigl( 1 - \No (\fp)^{-\kappa_{\fp}} \bigr) .
\]
This formula follows from an inclusion-exclusion argument via the
periodic covering sets that are obtained by taking only finitely
many primes into account, together with a tail estimate that
relies on \eqref{eq:G-cond}, as in Section~\ref{sec:extend}. 
The density formula simplifies to
$ \dens ( V_{\sg, \kappa} ) = 1/\zeta^{\pa}_{\QQ(\ii)} (\kappa)$ when all
$\kappa_{\fp}=\kappa\geqslant 2$, where $\zeta^{\pa}_{\QQ (\ii)}$ is the
Dedekind zeta function of $\QQ(\ii)$.

To describe $V_{\sg, \bs{\kappa}}$ as a weak model set, we can now
employ the CPS $(\CC, H, \cL)$ with $\CC\simeq \RR^2$ as direct space,
the compact Abelian group
$H = \bigotimes_{\fp\in\pg'} \ZZ[\ii]/\fp^{\kappa_{\fp}} \ZZ[\ii]$ as
internal space, and the diagonal embedding of $\ZZ[\ii]$ into
$\CC{\times} H$ as lattice, so
$\cL = \bigl\{ (x,x^{\star}): x\in \ZZ[\ii] \bigr\}$ with the
$\star$-map being given by
$x\mapsto x^{\star} \defeq (x \bmod \fp^{\kappa_{\fp}} )^{\pa}_{\fp \in
  \pg'}$.

The FB spectrum of $V^{\pa}_{\sg,\bs{\kappa}}$, and thus the dynamical
spectrum of the induced MTDS with the measure $\mu_{_\mathrm{M}}$, can
now be calculated as in the previous section. Following the route from
Remark~\ref{rem:dual-alternative}, and observing $\cO^* = \cO$
together with
$\bigl(\fp^{\kappa_{\fp}}\cO \cap \fq^{\kappa_{\fq}} \cO\bigr)^* =
(\ol{\fp})^{-\kappa_{\fp}} + (\ol{\fq})^{-\kappa_{\fq}}$ for any pair
$\fp \ne \fq$ of Gaussian primes (and analogously for multiple
intersections), one finds
\[
  L^{\circledast} \, = \, \Bigl\{ k \in \QQ[\ii] : 
  \begin{array}{c} \den (k) \text{ is 
    $(\kappa^{\pa}_{\ol{\fp}}\ts {+} 1)$-free for all }
    \fp \in \pg'  \\ \text{ and not divisible by any }
    \fq \in \pg \!\!\setminus\! \pg' \end{array}
    \Bigr\} \, = \, 
    \sum_{\fp\in\pg'} \bigl( \fp^{-\kappa^{\pa}_{\ol{\fp} }}\bigr) ,
\]
where $L^{\circledast}\!/\cO = \bigoplus_{\fp\in\pg} \bigl(
\fp^{-\kappa^{\pa}_{\ol{\fp} }}\bigr) /\cO$ is the direct sum of
countably many finite Abelian groups, all viewed as subgroups of
$\QQ(\ii)/\cO \subset \TT^2$.  This structure is also reflected in the
CPS and its dual, the latter being $(\CC, \widehat{H}, \cL^{0})$ with 
$\widehat{\CC}\simeq \CC$, the compact Abelian group
$\widehat{H} = \bigoplus_{\fp\in\pg'} \cO / (\fp^{\kappa_{\fp}})$,
where $\cO / (\fp^{\kappa_{\fp}} ) \simeq (\fp^{-\kappa_{\fp}})/\cO$,
and the annihilator $\cL^{0}$ of $\cL$ as its lattice. Each of the
contributing Abelian groups can be decomposed via
Fact~\ref{fact:factors}.

\begin{remark}
  Despite several similarities between the visible lattice points of
  $\ZZ^2$ and the square-free integers in $\ZZ[\ii]$, the
  corresponding dynamical systems are fundamentally different. From
  the above, we see that the spectra are different, so they cannot be
  isomorphic in the measure-theoretic sense, due to the Halmos--von
  Neumann theorem. Also as topological dynamical systems, they cannot
  be conjugate, because their extended symmetry groups differ
  \cite{BRY}, and none of them can be a factor of the other, by an
  argument put forward in \cite{BBN} that was later formulated in more
  general terms in \cite{Fabian}.  \exend
\end{remark}

\subsubsection{Eisenstein integers}

With $\xi = \frac{1}{2}(1 + \ii \sqrt{3}\,)$, which is a primitive
sixth root of unity, the ring of Eisenstein integers is $\ZZ[\xi]$,
which is the maximal order of $K=\QQ (\ii \sqrt{3}\,)$, another
Euclidean domain. The main difference to the previous example is the
fact that $\ZZ[\xi]$, viewed geometrically, is a triangular lattice in
$\RR^2$. All principal ideals are then triangular lattices as well.
The first step will be parallel to the above, but we will then take
another step to change it into a square lattice for direct comparison
with the $\ZZ^2$-action of symbolic dynamics.

The only ramified prime is $p=3$, where one has
$3 = \ol{\xi} (1+\xi)^2$. The inert primes are the ones with
$p\equiv 2 \bmod 3$, while $p\equiv 1 \bmod 3$ are the splitting
primes, then with $p = \fp \ol{\fp}$, such as
$7 = (2+\xi)(2+\ol{\xi}\ts )$ or $13=(3+\xi)(3+\ol{\xi}\ts )$. The
Eisenstein primes are thus
\[
  \cP_{_\mathrm{E}} \, = \, \{ 2, 1+\xi, 5, \fp^{\pa}_{(7)},
  \ol{\fp}^{\pa}_{(7)}, 11, \fp^{\pa}_{(13)},
  \ol{\fp}^{\pa}_{(13)}, 17, \fp^{\pa}_{(19)},
  \ol{\fp}^{\pa}_{(19)}, 23, \ldots \} \ts ,
\]
where we have
\begin{equation}\label{eq:Eisen-dual}
  \cO^{*} \ts = \, \myfrac{2 \ii}{\sqrt{3}\ts} \ts \cO \quad
  \text{and} \quad (z)^{*} \ts = \, \myfrac{1}{\ol{z}} \ts \cO^{*}
  \ts = \, \myfrac{2 \ii \ts z}{\sqrt{3} \ts\ts |z|^2} \ts \cO \ts .
\end{equation}
The required restrictions on $\bs{\kappa}$ are completely analogous to 
the case of the Gaussian integers. With 
${\pe}{\!'} = \{ \fp \in \pe : \kappa_{\fp} < \infty \}$,
this then results in the FB spectrum
\[
  L^{\circledast} \ts = \, \myfrac{2\ii}{\sqrt{3}\ts} \ts
    \Bigl\{ k \in \QQ(\xi) : \begin{array}{c} \den (k) \text{ is
    $(\kappa^{\pa}_{\ol{\fp}} +1)$-free for all $\fp \in \pe'$} \\ 
    \text{and not
    divisible by any $\fq \in \pe \!\setminus \nts \pe'$} 
    \end{array} \Bigr\} \, = \, \myfrac{2\ii}{\sqrt{3}\ts}
     \sum_{\fp\in\pe'} \bigl( \fp^{-\kappa^{\pa}_{\ol{\fp}}} \bigr) ,
\]
again computed via the dual lattice approach of 
Remark~\ref{rem:dual-alternative}, and the formula from
\eqref{eq:Eisen-dual}. Like before, we also get the structure
from Remark~\ref{rem:dual-discrete}, here with
\[
  L^{\circledast} \! / \cO^{*} \ts = \, \frac{ 2 \ii }{\sqrt{3} \ts}
  \bigoplus_{\fp\in\pe'} \bigl( \fp^{-\kappa^{\pa}_{\ol{\fp}}} \bigr)/\cO \ts ,
\]
which emerges from the restriction of $L^{\circledast}$ to the
fundamental domain for $\cO^{*}$. Again, the representation chosen
matches the needs of the dynamical context, while all contributing
groups, up to isomorphism, can be decomposed via
Fact~\ref{fact:factors}.

Next, we want to adjust this result for a change from $\cO$ to a
square lattice.  Observe that a lattice $\vG \subset \RR^2$ with basis
matrix $B$ is turned into the standard integer lattice by left
multiplication with $B^{-1}$, so $B^{-1} \vG = \ZZ^2$. Then, the dual
is given by
\[
    (B^{-1} \vG)^{*} \ts = \, B^{\trans} \vG^{*} .
\]
Here, the basis matrix of $\cO$ relative to the standard Cartesian
basis reads
\[
  B \, = \, \myfrac{1}{2} \begin{pmatrix} 2 & 1 \\
    0 & \sqrt{3\ts} \end{pmatrix} .
\]
Consequently, if we replace $\vG = \cO$ by $B^{-1} \vG$, the FB
spectrum changes to $B^{\trans} L^{\circledast}$, again with $z=x+\ii y$
being identified with the column vector $(x,y)^{\trans}$.

\subsubsection{General imaginary quadratic field}

Let $K$ be a general imaginary quadratic field, with maximal order
$\cO = \cO^{\pa}_{K}$. We can always view $\cO$ as a lattice in
$\RR^2 \simeq \CC$. Here, as the class number need not be $1$, we have
to work with prime ideals and their powers. We simply write $\fp$ for
a prime ideal, and $\cP_{\cO}$ for the set of all prime ideals. 
Given $\kappa_{\fp} \in\NN \cup \{ \infty \}$ for
$\fp\in\cP_{\cO}$ subject to the condition
\begin{equation}\label{eq:gen-cond}
  \sum_{\fp\in\cP^{\pa}_{\!\cO}} [\cO : \fp]^{-\kappa_{\fp}} \, < \, \infty \ts ,
\end{equation}
we define
$V = \cO \setminus \bigcup_{\fp\in\cP^{\prime}_{\!\cO}}
\fp^{\kappa_{\fp}}$, where
$\cP^{\ts\prime}_{\!\cO} = \{ \fp \in \cP^{\pa}_{\! \cO} : \kappa_{\fp} <
\infty \}$.  The denominator of a number $k\in K$ is defined as
before, now with coprimality understood in terms of ideals.

\begin{theorem}\label{thm:spectrum-imag}
  For\/ $\delta \in \NN$ square-free, consider the imaginary quadratic
  field\/ $K = \QQ (\ii \sqrt{\delta\ts}\,)$ with its ring of
  integers, $\cO = \cO^{\pa}_{K}$. Let\/ $\cP^{\pa}_{\!\cO}$ be the set of
  prime ideals in\/ $\cO$, select\/
  $\bs{\kappa} = (\kappa_{\fp})^{\pa}_{\fp \in \cP^{\pa}_{\cO}}$ with\/
  $\kappa_{\fp} \in \NN \cup \{ \infty \}$ subject to the condition in
  \eqref{eq:gen-cond}, and let\/ $\cP^{\ts\prime}_{\cO}$ be as above.

  Then, the FB spectrum of\/
  $V = \cO \setminus \bigcup_{\fp\in\cP^{\prime}_{\!\cO}}
  \fp^{\kappa_{\fp}}$ is given by
\[
  L^{\circledast} \ts = \, \myfrac{\ii \psi^{\pa}_{\delta}}
     {\sqrt{\delta}\ts} \ts \Bigl\{ k \in K :
     \begin{array}{c} \den (k) \text{ is
       $(\kappa^{\pa}_{\ol{\fp}} \nts + \nts 1)$-free for all
       $\fp \in \cP^{\ts\prime}_{\!\cO}$} \\ 
       \text{and not divisible by any
       $\fq \in \cP_{\!\cO} \nts\setminus \nts \cP^{\ts\prime}_{\!\cO}$} 
    \end{array} \Bigr\} \, = \,
    \myfrac{\ii \psi^{\pa}_{\delta}}{\sqrt{\delta}\ts}
    \sum_{\fp\in\cP^{\prime}_{\!\cO}} \ol{\fp}^{\ts -\kappa^{\pa}_{\fp}}  ,
\]
with\/ $\psi^{\pa}_{\delta} = 2$ for\/ $\delta\equiv 3 \bmod 4$ and\/
$\psi^{\pa}_{\delta}=1$ otherwise. Furthermore, one has
\[
  L^{\circledast}\!/\cO^{*} \, = \, \myfrac{\ii
    \psi^{\pa}_{\delta}}{\sqrt{\delta}\ts}
  \bigoplus_{\fp\in\cP^{\prime}_{\!\cO}} \ol{\fp}^{\ts
    -\kappa^{\pa}_{\fp}} \! /\cO
\]
as the restriction of\/ $L^{\circledast}$ to the fundamental domain 
of\/ $\cO^{*}\!$.
\end{theorem}

\begin{proof}
  In view of Remark~\ref{rem:dual-alternative}, any element of
  $L^{\circledast}$ lies in the dual of
  $\bigcap_{\fp\in \cI} \fp^{\kappa_{\fp}}$ for some \emph{finite}
  $\cI \subset \cP^{\pa}_{\!\cO}$. As the prime ideals are mutually
  coprime, we get
\[
  \Bigl( \bigcap_{\fp\in\cI} \fp^{\kappa_{\fp}}\Bigr)^{*} = \,
  \sum_{\fp\in\cI} \bigl( \fp^{\kappa_{\fp}} \bigr)^{*} = \,
  \sum_{\fp\in\cI} \ol{\fp}^{\ts -\kappa_{\fp}} \cO^{*} = \,
  \myfrac{\ii \psi^{\pa}_{\delta}}{\sqrt{\delta}\ts}
  \sum_{\fp\in\cI} \ol{\fp}^{\ts -\kappa_{\fp}} 
\]
via Eq.~\eqref{eq:dual-ideal-complex}, where
$\cO^{*} = \frac{\ii \psi^{\pa}_{\delta}}{\sqrt{\delta\ts}}\ts \cO$ by
Lemma~\ref{lem:dual-order-complex}.  As $\cI$ can be any finite 
subset of $\cP^{\ts\prime}_{\!\cO}$, the formula for $L^{\circledast}$ 
follows. A standard computation modulo $\cO^{*}$ then results 
in the claimed representation of $L^{\circledast}\!/\cO^{*}$
as a direct sum.
\end{proof}  

Turning $\cO$ into $\ZZ^2$ as explained before, we get the
corresponding result for the symbolic dynamical system as follows.

\begin{coro}
  Let the setting be as in Theorem~$\ref{thm:spectrum-imag}$.  If\/
  $B$ is the basis matrix of\/ $\cO$, one has\/ $B^{-1} \cO = \ZZ^2$,
  and the FB spectrum of the\/ $\bs{\kappa}$-free integers in this
  formulation is given by\/ $B^{\trans} L^{\circledast}$.  This is
  also the dynamical spectrum of the MTDS\/
  $(\YY^{\pa}_V, \RR^2, \mu_{_\mathrm{M}})$, while\/ its intersection
  with\/ $ \TT^2 = \RR^2 \! /\ZZ^2$ is the one of the MTDS\/
  $(\XX^{\pa}_V, \ZZ^2, \mu_{_\mathrm{M}})$.  \qed
\end{coro}

\subsection{Real quadratic fields}

Unlike for imaginary fields, the maximal order $\cO = \cO^{\pa}_{K}$ of a
real quadratic field $K = \QQ(\sqrt{d\ts}\,)$ is a dense subset of
$\RR$, but not a lattice.  It becomes one under the standard Minkowski
embedding $ \theta \colon \cO \xrightarrow{\quad} \RR^2$, as defined
by $x \mapsto (x, x')$ with ${}'$ being algebraic conjugation in
$K$. The latter is the unique field automorphism that fixes $\QQ$ and
sends $\sqrt{d}\mapsto - \sqrt{d}$. This gives
$\cL = \theta (\cO) = \{ (x, x') : x \in \cO \}$, which is a lattice
in $\RR^2$. Clearly, the mapping $\theta$ is also well defined on $K$.

To continue, we need the symmetric $\QQ$-bilinear form
\begin{equation}\label{eq:bilinear-real}
  x.y \, \defeq \, \tr (xy) \, = \, xy + x' y'
\end{equation}
which equals the standard scalar product of $\theta (x)$ with
$\theta(y)$ in $\RR^2 \nts$. It is related to the field norm
$N(x) = x x'$ via the estimate
\[
    x.x \, \geqslant \, 2 \ts\ts \lvert N (x) \rvert \ts ,
\]
as follows from $x.x = (x \mp x' )^2 \pm 2 x x'$ via a case
distinction whether $xx'$ is positive or negative. This will later
again allow us to use Fact~\ref{fact:norm-est}.

Relative to \eqref{eq:bilinear-real}, the \emph{dual} of $\cO$ reads
\[
  \cO^* \, = \, \{ y \in K : x.y \in \ZZ \text{ for all } x \in \cO \}
  \ts ,
\]
which agrees with the co-different of $\cO$ relative to $\ZZ$ here and
can be calculated as a special case of
\cite[Prop.~III.2.4]{Neukirch}. In our setting, this can be stated as
follows.
  
\begin{lemma}\label{lem:dual-order-real}
  Let\/ $d > 1$ be a square-free integer and consider the
  corresponding real quadratic field, $K = \QQ(\sqrt{d\ts}\,)$, with
  its ring of integers, $\cO = \cO^{\pa}_{K}$, according to
  Eq.~\eqref{eq:order}. Then, its dual with respect to the\/
  $\QQ$-bilinear form \eqref{eq:bilinear-real} is
\[
    \cO^{*} \, = \, \myfrac{1}{2 \sqrt{d\ts}} \ts \begin{cases}
    2 \ts \cO , &
    \text{if } d \equiv 1 \bmod 4 \ts , \\
    \cO , & \text{otherwise}\ts .
  \end{cases}
\]  
\end{lemma}

\begin{proof}
  When $d \not\equiv 1 \bmod 4$, the maximal order is
  $\cO = \langle 1, \sqrt{d\ts}\, \rangle^{\pa}_{\ZZ}$, where
  $\bigl\{ \frac{1}{2}, \frac{1}{2 \sqrt{d\ts}} \bigr\}$ is the dual
  basis with respect to \eqref{eq:bilinear-real}, whence we obtain
\[
  \cO^{*} \ts = \, \Big\langle \myfrac{1}{2}, \myfrac{1}{2\sqrt{d\ts}}
  \Big\rangle_{\ZZ} \, = \, \myfrac{1}{2\sqrt{d\ts}} \big\langle
  \sqrt{d\ts}, 1 \big\rangle^{\pa}_{\ZZ} \, = \, \myfrac{1}{2\sqrt{d\ts}}
  \big\langle 1 , \sqrt{d\ts}\, \big\rangle^{\pa}_{\ZZ} \, = \,
  \myfrac{1}{2\sqrt{d\ts}} \ts \cO \ts .
\]  
  
Likewise, when $d \equiv 1 \bmod 4$, we have
$\cO = \big\langle 1, \frac{1+\sqrt{d\ts}}{2} \big\rangle_{\ZZ}$ with
$\big\{ \frac{\sqrt{d\ts}-1}{2 \sqrt{d\ts}} , \frac{1}{\sqrt{d\ts}}
\big\}$ as dual basis relative to \eqref{eq:bilinear-real}. This
results in
\[
  \cO^{*} \ts = \, \myfrac{1}{\sqrt{d\ts}} \ts \Big\langle
  \myfrac{\sqrt{d} \ts {-} 1}{2} , 1 \Big\rangle^{\pa}_{\ZZ} \, = \,
  \myfrac{1}{\sqrt{d\ts}} \ts \Big\langle 1,
  \myfrac{1{+}\sqrt{d\ts}}{2} \Big\rangle_{\ZZ} \, = \,
  \myfrac{1}{\sqrt{d\ts}} \ts\ts \cO
\]
and proves the remaining case.
\end{proof}

Let us mention in passing that the basis matrices of $\theta (\cO)$
with respect to the standard Cartesian basis of $\RR^2$ read
\begin{equation}\label{eq:basis-real}
  B \, = \, \begin{pmatrix} 1 & \sqrt{d\ts}\, \\
    1 & {-}\sqrt{d\ts} \end{pmatrix}  \quad \text{and} \quad
    B \, = \, \begin{pmatrix} 1 & \frac{1+\sqrt{d\ts}}{2}\, \\
    1 & \frac{1-\sqrt{d\ts}}{2} \end{pmatrix}
\end{equation}
for $d\not \equiv 1 \bmod 4$ and $d\equiv 1 \bmod 4$,
respectively.

Now, Lemma~\ref{lem:dual-order-real} also gives us the dual of any
non-trivial principal ideal, namely
\begin{equation}\label{eq:gen-id}
  (x)^* \ts = \, \myfrac{1}{x}\ts \cO^* \ts = \,
  \myfrac{\psi^{\pa}_{d}}{x} \ts \cO \, = \,
  \myfrac{\psi^{\pa}_{d} }{N(x)} (x') \ts ,
\end{equation}
where $\psi^{\pa}_{d} = 1/\sqrt{d\ts}$ for $d\equiv 1 \bmod 4$ and half
of this number for all other cases.  For a general non-trivial ideal
$\fa$, we get the dual as
\[
     \fa^{*} \ts = \, \fa^{-1} \ts \cO^{*} \, = \, \psi^{\pa}_{d} \, 
     \fa^{-1} \ts \cO \ts ,
\]
again with the inverse ideal $\fa^{-1}$.

Cases with class number $1$ (in the wider sense) include
$\QQ (\sqrt{d\ts}\,)$ with
\[
  d \in \{ 2^*, 3, 5^*, 6, 7, 11, 13^*, 14, 17^*, 19, 21, 22,
           23, 29^*, 31, ... \} \ts ,
\]
where it is still unknown whether this set is infinite (as
conjectured) or not. Integers with a ${}^*$ mark cases whose class
number is $1$ also in the narrow sense \cite{Zagier}; see Sequences
\textsf{A{\ts}003172} and \textsf{A{\ts}003655} of \cite{OEIS}. Let us
consider two special cases first.

\subsubsection{The case $d=2$}

Let us treat $K=\QQ(\sqrt{2\ts}\,)$ and its maximal order,
$\cO=\ZZ[\sqrt{2}\,]$, which is one of the cases with class number
$1$. Lemma~\ref{lem:dual-order-real} and Eq.~\eqref{eq:gen-id} give
\[
  \cO^{*} \ts = \, \myfrac{1}{2 \sqrt{2\ts}} \ts \cO \, = \, 
  \myfrac{1}{4} \ts (\sqrt{2\ts} \, ) \quad \text{and} \quad
  (x)^{*} \ts = \, \myfrac{1}{2 \sqrt{2\ts}\ts x} \ts \cO \ts .
\]

The primes of $\cO$ are again of three types and emerge from the
rational primes as follows. There is one ramified prime, which is
$2 = (\sqrt{2}\,)^2$, while all rational primes
$p\equiv \pm 3 \bmod 8$ are inert. The splitting primes are the
rational primes with $p \equiv \pm 1 \bmod 8$, and any such prime
splits as $p = (r + s \sqrt{2} \, )(r - s \sqrt{2}\,)$, so
$p=r^2 - 2 s^2$, with coprime $r,s \in \ZZ$. We thus get the list of
primes in $\cO$ as
\[
  \cP^{\pa}_{\cO} \, = \, \{ \sqrt{2}, 3, 5, \fp^{\pa}_{(7)},\fp^{\ts\prime}_{(7)},
  11, 13, \fp^{\pa}_{(17)}, \fp^{\ts\prime}_{(17)}, 19, \ldots \} \ts ,
\]
again listed along increasing rational primes. Note that the absolute
norm for elements of $K$ is given by
$\No ( a + b \sqrt{2}\, ) = \lvert a^2-2 \ts b^2 \rvert$.

With $\cP$ and $\cP'$ defined in analogy to above, the FB spectrum is
\[
  L^{\circledast} \ts = \, \myfrac{1}{2\sqrt{2}\ts}\ts
    \Bigl\{ k \in K : \! \begin{array}{c} \den (k) \text{ is
         $(\kappa_{\fp} \nts + \nts 1)$-free for all
         $\fp \in \cP^{\ts\prime}_{\cO}$} \\
    \text{and not divisible by any
         $\fq \in \cP^{\pa}_{\nts\cO} \! \setminus \nts\nts
           \cP^{\ts\prime}_{\nts\cO}$} 
    \end{array} \! \Bigr\} \, = \, \myfrac{1}{2\sqrt{2}\ts}
    \sum_{\fp \in \cP^{\ts\prime}_{\nts\cO}} 
    \bigl( \fp^{-\kappa_\fp} \bigr)  ,
\]
again calculated via the dual lattices and
Remark~\ref{rem:dual-alternative}.  One can now turn $\theta (\cO)$
into $\ZZ^2$ by left multiplication with $B^{-1}$ as before, and the
FB spectrum then becomes $B^{\trans} L^{\circledast}$.

\subsubsection{The case $d=5$}

For $K = \QQ(\sqrt{5\ts}\,)$, one has $\cO=\ZZ[\tau]$, where
$\tau = \frac{1}{2}(1{+}\ts\sqrt{5\ts}\,)$ is the golden ratio. Here,
Lemma~\ref{lem:dual-order-real} and Eq.~\eqref{eq:gen-id} give 
\[
      \cO^{*}  \ts = \, \myfrac{1}{\sqrt{5}\ts} \ts \cO \, = \,
      \myfrac{1}{5} \ts (\sqrt{5\ts}\,) \quad \text{and} \quad
      (x)^{*} \ts = \, \myfrac{1}{\sqrt{5\ts}\ts x} \ts \cO \ts .
\]    
The primes of $\cO$ emerge from $\cP$ as $5 = (2 \tau {-} 1)^2$, which
is ramified, while any $p \equiv \pm 2 \bmod 5$ is inert and
$p\equiv \pm 1 \bmod 5$ splits, the latter as
$p = (r + s \tau) (r + s \tau' )$, so $p = r^2 + rs - s^2$, hence
\[
  \cP^{\pa}_{\cO} \, = \, \{ 2, 3, \sqrt{5}, 7, \fp^{\pa}_{(11)},
  \fp^{\ts\prime}_{(11)}, 13, 17, \fp^{\pa}_{(19)},
  \fp^{\ts\prime}_{(19)}, \ldots \} \ts .
\] 
The rest is completely analogous to the previous example.

\subsubsection{General real quadratic field}

The general situation is analogous to the imaginary case, except that
we need the Minkowski embedding $\theta$ to obtain lattices in
$\RR^2$. We can thus state the result as follows.

\begin{theorem}\label{thm:spectrum-real}
  For\/ $1 < d \in \NN$ square-free, consider the real quadratic
  field\/ $K = \QQ (\sqrt{d\ts}\,)$ with its ring of integers,
  $\cO = \cO^{\pa}_{K}$. Let\/ $\cP^{\pa}_{\!\cO}$ denote the set of prime
  ideals in\/ $\cO$, let\/
  $\bs{\kappa} = (\kappa_{\fp})^{\pa}_{\fp \in \cP^{\pa}_{\cO}}$ with\/
  $\kappa_{\fp} \in \NN \cup \{ \infty \}$, subject to the condition
  in \eqref{eq:gen-cond}, and set\/ $\cP^{\prime}_{\cO} =
  \{ \fp \in \cP^{\pa}_{\cO} : \kappa_{\fp} < \infty \}$.

  Then, the FB spectrum of\/
  $V = \theta \bigl(\cO \setminus \bigcup_{\fp\in\cP^{\prime}_{\!\cO}}
  \fp^{\kappa_{\fp}} \bigr)$ is given by
\[
  L^{\circledast} \ts  = \, \theta \biggl( \myfrac{ \psi^{\pa}_{d}}
     {2 \sqrt{d \ts}\ts} \ts \Bigl\{ k \in K :
     \begin{array}{c} \den (k) \text{ is
       $(\kappa_{\fp} \nts + \nts 1)$-free for all
       $\fp \in \cP^{\ts\prime}_{\!\cO}$} \\
       \text{and not divisible by any
       $\fq \in \cP_{\!\cO} \nts\setminus \nts \cP^{\ts\prime}_{\!\cO}$} 
    \end{array} \Bigr\} \biggr) \, = \,
    \myfrac{\psi^{\pa}_{d}}{2 \sqrt{d \ts}\ts} \,
    \theta \Bigl( \sum_{\;\fp\in\cP^{\prime}_{\!\cO}} \fp^{-\kappa_{\fp}} \Bigr) ,
\]
with\/ $\psi^{\pa}_{d} = 2$ for\/ $d \equiv 1 \bmod 4$ and\/
$\psi^{\pa}_{d}=1$ otherwise.  \qed
\end{theorem}

The situation after changing to (geometric) $\ZZ^2$- and
$\RR^2$-action is as follows.

\begin{coro}
  Let the setting be as in Theorem~$\ref{thm:spectrum-real}$.  If\/
  $B$ is the basis matrix of the Minkowski embedding\/ $\theta(\cO)$
  according to Eq.~$\eqref{eq:basis-real}$, one has\/
  $B^{-1} \theta (\cO) = \ZZ^2$, and the FB spectrum of the\/
  $\bs{\kappa}$-free integers in the formulation with\/ $U=B^{-1} V$ is 
  given by\/ $B^{\trans} \theta (L^{\circledast})$.
  
  This is also the dynamical spectrum of the MTDS\/
  $(\YY^{\pa}_{\! U}, \RR^2, \mu_{_\mathrm{M}})$, while\/ its
  intersection with\/ $ \TT^2 = \RR^2\! / \ZZ^2$ is the one of the
  MTDS\/ $(\XX^{\pa}_U, \ZZ^2, \mu_{_\mathrm{M}})$, where one obtains
\[
     \TT^2 \cap B^{\trans} \theta (L^{\circledast}) \, = \,
     \myfrac{\psi^{\pa}_{d}}{2 \sqrt{d \ts} \ts} \, B^{\trans} \theta 
     \Bigl( \bigoplus_{\;\fp\in\cP^{\prime}_{\!\cO}} \fp^{-\kappa_{\fp}}
     \! / \cO    \Bigr)
\]   
   in generalisation of our previous expressions.  \qed
\end{coro}

This corollary illustrates the role of the explicit representation for
the dynamics. Up to group isomorphism, we get the same spectra as in
Theorem~\ref{thm:spectrum-real}, but the representations are rather
different. This is precisely what the Halmos--von Neumann theorem uses
to distinguish ergodic measure-theoretic dynamical systems with
pure-point spectrum up to measurable conjugacy. When they fail to be
conjugate, this can be caused by a linear transform (as here), but it
can also be more substantial (and thus more relevant).

While we have treated a particular class of examples, the structure is
sufficiently robust to allow for various generalisations. Also, the
role of different ergodic measures can certainly be explored further,
where we expect the FB coefficients to feature prominently. It is
presently open to what extent the more detailed results can be
transferred from the number-theoretic setting to the much more general
lattice setting.

\appendix
\section{Bernoulli thinning of uniformly distributed 
sequences}\label{sec:app}

The purpose of this appendix is a brief, self-contained exposition of
the equidistribution result for random subsequences of a uniformly
distributed sequence, adapted to our needs in the main text. The
result as such is well known, but not so easy to locate in the
literature; see \cite{KN} for general background and \cite{Rubel} (and
references therein) for a more specific treatment in the context of
locally compact groups.

Below, we shall first use the standard approach to uniform
distribution of sequences in $[0,1]$, and then its generalisation to
bounded sets $U\subset \RR$ (or $U \subset \RR^d$) such that
$\mathbf{1}^{\pa}_{U}$ is Riemann integrable, which is the case if and
only if $\partial U$ has zero Lebesgue measure.  Such sets are also
called \emph{Jordan measurable}. There is a well-known generalisation
to (relatively) compact subsets of second countable, locally compact
Abelian groups, but we shall also need uniform distribution in compact
sets with a `fat' boundary (of positive measure).

Let us begin with the unit interval, compare \cite[Ch.~1]{KN}, and let
the sequence $(a^{\pa}_{m})^{\pa}_{m\in \NN}$ be uniformly distributed in
$[0,1]$, which is to say that, for all
$0\leqslant \alpha < \beta \leqslant 1$, one has
\begin{equation}\label{app:eq-1}
    \lim_{n\to\infty} \myfrac{1}{n} 
    \sum_{m=1}^{n} \bs{1}^{\pa}_{[\alpha,\beta]}
    (a^{\pa}_{m}) \, = \, \beta-\alpha \ts .
\end{equation}
Now, we are interested in the question what happens if we go to a
subsequence $\bigl( a^{\pa}_{m_i}\bigr)_{i\in\NN}$ that emerges from a
random selection of elements. The standard formulation employs a
family $(\xi^{\pa}_{i})^{\pa}_{i\in\NN}$ of i.i.d.\ Bernoulli random
variables, which take values in $\{ 0,1\}$ with common distribution
defined by $\PP (\xi^{\pa}_{1} = 1) = p \in (0,1)$. Then, each
realisation gives rise to such a subsequence by keeping the $a^{\pa}_{m}$
where $\xi^{\pa}_{m} =1$. We call this \emph{Bernoulli thinning}. We now
need the \emph{strong law of large numbers} (SLLN) in several forms.

Let $0\leqslant \alpha < \beta \leqslant 1$ be arbitrary, but fixed.
One can define a new family of random variables via
$X^{\pa}_m = \xi^{\pa}_m \bs{1}^{\pa}_{[\alpha,\beta]} (a^{\pa}_{m})$, which are
independent, but no longer identically distributed. Each $X_n$ is
either identically $0$ or takes values in $\{ 0, 1\}$, so the variance
is always bounded by $1$. Then, Kolmogorov's version of the SLLN
applies, compare \cite[Thm.~14.5]{Bauer}, and we get
\[
     \lim_{n\to\infty} \myfrac{1}{n} 
    \sum_{\ell=1}^{n} \bigl( X^{\pa}_{\ell} - \EE (X^{\pa}_{\ell}) \bigr)
     \, = \, 0 \ts , 
\]
which holds almost surely, that is, for almost all realisations of the
thinning. Since $\EE (X^{\pa}_{\ell}) = p$ if
$a^{\pa}_{\ell} \in [\alpha, \beta]$ and $\EE (X^{\pa}_{\ell}) = 0$
otherwise, using \eqref{app:eq-1}, one has
\[
    \myfrac{1}{n} \sum_{\ell=1}^{n} \EE (X^{\pa}_{\ell}) 
    \, = \, \myfrac{p}{n} \sum_{\ell=1}^{n}
    \bs{1}^{\pa}_{[\alpha, \beta]} (a^{\pa}_{\ell})
    \, \xrightarrow{\, n\to\infty\,} \, p \ts (\beta - \alpha) \ts .
\]
Observing that $\frac{1}{n} \sum_{i=1}^{n} \xi^{\pa}_{i} $ almost surely
converges to $p$, by the standard version of the SLLN, we get, almost
surely as $n\to\infty$, the asymptotic behaviour
\[
    \myfrac{1}{n} \sum_{m=1}^{n} X^{\pa}_{m}
    \, \sim \, \frac{p}{\sum_{i=1}^{n} \xi^{\pa}_{i}}
    \sum_{\substack{m=1_{\vphantom{a}} \\ \xi^{\pa}_{m} = 1}}^{n}
    \bs{1}^{\pa}_{[\alpha,\beta]} (a^{\pa}_{m}) \, = \,
    \myfrac{p}{\#^{\pa}_n} \sum_{i=1}^{\#^{\pa}_n} \bs{1}^{\pa}_{[\alpha,\beta]}
    (a^{\pa}_{m_i}) \ts ,
\]
where $\#^{\pa}_n$ is the (random) number of elements of the original
sequence $(a^{\pa}_{1}, a^{\pa}_{2}, \ldots , a^{\pa}_{n})$ that remain after
the Bernoulli thinning and $(m_i)^{\pa}_{i\in\NN}$ is the sequence of
integers for which $\xi^{\pa}_{m_i}=1$. Note that, almost surely as
$n\to\infty$, one has $\#^{\pa}_n \sim p \ts n$, again by the SLLN.

Putting the two pieces together, we see that
\[
    \lim_{N\to\infty} \myfrac{1}{N} \sum_{i=1}^{N} 
    \bs{1}^{\pa}_{[\alpha, \beta]} (a^{\pa}_{m_i}) 
    \, = \, \beta - \alpha \ts ,
\]
which holds almost surely in the above sense.

By a standard countability argument, we then also know the
simultaneous almost sure validity of the previous limit for all
$\alpha, \beta \in \QQ \cap [0,1]$, which is dense in $[0,1]$.

Finally, monotonicity in conjunction with inner and outer regularity
of Lebesgue measure gives us the almost sure equidistribution as
follows.

\begin{theorem}
  Let the sequence\/ $(a^{\pa}_{n})^{\pa}_{n\in\NN}$ be uniformly
  distributed in\/ $[0,1]$. Then, almost every random subsequence
  obtained from a Bernoulli thinning with probability\/ $p\in (0,1)$
  is still uniformly distributed in the unit interval.  \qed
\end{theorem}

It is well known \cite[Sec.~1.1.5]{KN} that the limit in
Eq.~\eqref{app:eq-1}, by the regularity of Lebesgue measure
$\lambda_{_\mathrm{L}}$, has the generalisation
\[
   \lim_{n\to\infty} \myfrac{1}{n} \sum_{\ell=1}^{n}
   \bs{1}^{\pa}_{W} (a^{\pa}_{\ell}) \, = \, \lambda_{_\mathrm{L}} (W) \ts ,
\]
to any Jordan-measurable set $W\subseteq [0,1]$, and, more generally,
also to any Jordan-measurable set $W\subseteq [0,1]^d\subset \RR^d$,
when the sequence $(a^{\pa}_n)^{\pa}_{n\in\NN}$ is uniformly distributed in
$[0,1]^d$; consult \cite[Sec.~1.6]{KN} for further details.

There is a generalisation of the entire concept to compact subsets of
a locally compact Abelian group $G$ that is second countable (which
means that is has a countable basis for its topology); see
\cite[Ch.~3.2]{KN} as well as \cite{Rubel} and references therein. To
keep things simple, we note that we only need this generalisation for
compact Abelian groups of the product form
$H = \bigotimes_{i\in\NN} G_i$ where each $G_i$ is \emph{finite}
Abelian. Here, second countability of $H$ follows
constructively. Indeed, since each $G_i$ has only finitely many
subsets and since the finite subsets of $\NN$ are countable, the
collection
\[
  \bigl\{ (U_i)_{i\in\NN} : U_i \subseteq G_i \text{ with }
   U_i = G_i \text{ for all but finitely many } i \bigr\}
\]
defines a countable basis for the topology of $H$. We abbreviate the
product set of the $U_i$ as
$(U^{\pa}_{1}, U^{\pa}_{2}, \ldots , U^{\pa}_{N})$ if $U_i = G_i$
holds for all $i > N$, and agree to use the smallest index $N$ with
this property. These sets play the role of cylinder sets.

We now assume $G$ to be equipped with its normalised Haar measure
$\nh$, so $\nu^{\pa}_{\nts H} (H) =1$. Clearly, we then get the measure
of a cylinder set as
\[
    \nh (U^{\pa}_{1}, U^{\pa}_{2}, \ldots , U^{\pa}_{N}) \, = 
    \prod_{i=1}^{N} \frac{| U_i |}{| G_i |} \ts ,
\]
where $|.|$ denotes the cardinality of a finite set.

We can now define uniform distribution of a sequence
$(a_n)^{\pa}_{n\in\NN}$ in $H$ or in compact subsets of $H$ as follows.

\begin{definition}\label{def:gen-dist}
  Let\/ $H=\bigotimes_{i\in\NN} G_i $ be the infinite product of
  finite Abelian groups, $W\subseteq H$ a compact subset with
  $\nh (W) >0$, and\/ $(a_n)^{\pa}_{n\in\NN}$ a sequence with values in\/
  $W\!$. Then, this sequence is called\/ \emph{uniformly distributed}
  in\/ $W$ if
\[
  \lim_{n\to\infty} \myfrac{1}{n} \sum_{m=1}^{n} \mathbf{1}^{\pa}_{C}
  (a^{\pa}_{m}) \, = \, \nh (W\nts\cap C)
\]  
holds for every cylinder set\/
$C = (U^{\pa}_{1}, U^{\pa}_{2}, \ldots , U^{\pa}_{N})$ of the type
introduced above.
\end{definition}

Then, our previous thinning argument can be applied to any set of the
form $W\nts \cap C$, of which there are countably many, and we get the
following result.

\begin{coro}\label{coro:ud-in-W}
  Let\/ $H = \bigotimes_{i\in\NN} G_i$ with all\/ $G_i$ being finite
  Abelian groups, and let\/ $W\subseteq H$ be a compact subset with\/
  $\nh (W)>0$.  Further, let\/ $(a^{\pa}_{m})^{\pa}_{m\in\NN}$ be a sequence
  with values in\/ $W$ that is uniformly distributed in\/ $W\!$. Then,
  almost every random subsequence\/ $(a^{\pa}_{m_i})^{\pa}_{i\in\NN}$
  obtained from a Bernoulli thinning with probability\/ $p\in (0,1)$
  is still uniformly distributed in\/ $W\!$, in the sense of
  Definition~$\ref{def:gen-dist}$. \qed
\end{coro}

It is clear that several immediate generalisations are possible for
going beyond finite factors, but we leave further details in this
direction to the interested reader.

If we are in the situation of Corollary~\ref{coro:ud-in-W}, we know
that the sequence of probability measures defined by
$\frac{1}{n} \sum_{j=1}^{n} \delta_{a_j}$ weakly converges to the
probability measure $\frac{1}{\nh (W)} \, \nh \nts \nts\big|_{W}$.

Uniform distribution of a sequence $(a_m)^{\pa}_{m\in\NN}$ in a compact
set $W$ means that $\frac{1}{n}\sum_{m=1}^{n} a_m$, as $n\to\infty$,
converges to $\int_{W} \dd \nh = \nh (W) = \nh(\bs{1}^{\pa}_{W})$. More
generally, for every continuous function $f$ on $H$, we also get
\[
   \lim_{n\to\infty} \myfrac{1}{n} \sum_{m=1}^{n} f (a_m) \, 
   = \int_{W} f(x) \dd \nh (x) \ts ,
\]
provided the sequence is uniformly distributed in $W$. In particular,
this holds for all characters on $H$, which are the continuous group
homomorphisms from $H$ into the unit circle. This result then allows
to determine the FB coefficients of model sets, including weak model
sets of maximal density like the ones in this paper, on the basis
of the result from \cite{Moody}.

\section*{Acknowledgements}

It is our pleasure to thank Philipp Gohlke, Markus Kirschmer,
Christoph Richard, Nicolae Strungaru and Vitali Wachtel for helpful
discussions, as well as Fabian Gundlach, J\"{u}rgen Kl\"{u}ners, Neil
Ma\~{n}ibo, Jan Maz\'{a}\v{c}, Andreas Nickel and Timo Spindeler for 
useful hints on the manuscript.

This work was supported by the German Research Council (Deutsche
Forschungsgemeinschaft, DFG) under SFB-TRR 358/1 (2023) -- 491392403
(MB, DL) and by the Austrian Science Fund FWF: P-33943-N (TS), by a
grant from the priority research area SciMat under the Strategic
Programme Excellence Initiative at Jagiellonian University (TS), and
by the MSCA individual fellow project ErgodicHyperbolic - 101151185
(TS).  

\end{document}